\documentclass{amsart}

\usepackage{amssymb,latexsym,amsmath,extarrows}
\usepackage{graphicx}

\newtheorem*{acknowledgement}{Acknowledgements}
\newtheorem{theorem}{Theorem}[section]
\newtheorem{lemma}[theorem]{Lemma}
\newtheorem{proposition}[theorem]{Proposition}

\newtheorem{definition}[theorem]{Definition}

\newcommand{\eps}{\epsilon}

\newcommand{\wh}{\widehat}

\newcommand{\ZR}{\mathbb{R}}

\newcommand{\bB}{{\bf B}}

\newcommand{\bT}{{\bf T}}

\newcommand{\eit}{e^{i t \Delta}}
\newcommand{\eir}{e^{i r \Delta}}
\newcommand{\hichi}{\raisebox{0.7ex}{\(\chi\)}}

\begin{document}

\title{A sharp Schr\"odinger maximal estimate in $\ZR^2$}

\author{Xiumin Du}
\address{Mathematics Department\\
University of Illinois at Urbana-Champaign\\
Urbana IL}
\email{xdu7@illinois.edu}
\author{Larry Guth}
\address{Department of Mathematics\\
MIT\\
Cambridge MA}
\email{lguth@math.mit.edu}
\author{Xiaochun Li}
\address{Mathematics Department\\
University of Illinois at Urbana-Champaign\\
Urbana IL}
\email{xcli@math.uiuc.edu}

\date{\today}

\begin{abstract}
We show that $\lim_{t \to 0} e^{it\Delta}f(x) = f(x)$ almost everywhere for all $f \in H^s (\mathbb{R}^2)$ provided that $s>1/3$. This result is sharp up to the endpoint. The proof uses polynomial partitioning and decoupling.
\end{abstract}

\maketitle

\section{Introduction}
\setcounter{equation}0

The solution to the free Schr\"{o}dinger equation
\begin{equation}
  \begin{cases}
    iu_t - \Delta u = 0, &(x,t)\in \mathbb{R}^n \times \mathbb{R} \\
    u(x,0)=f(x), & x \in \mathbb{R}^n
  \end{cases}
\end{equation}
is given by
$$
  e^{it\Delta}f(x)=(2\pi)^{-n}\int e^{i\left(x\cdot\xi+t|\xi|^2\right)}\widehat{f}(\xi) \, d\xi.
$$

We consider the following problem posed by Carleson in \cite{lC}: determine the optimal $s$ for which $\lim_{t \to 0}e^{it\Delta}f(x)=f(x)$ almost everywhere whenever
$f\in H^s(\mathbb{R}^n).$  Our main result is the following:

\begin{theorem}\label{Thm:PointConv}
 For every  $f\in H^s(\mathbb R^2)$ with $s>1/3$, $\lim_{t \to 0}e^{it\Delta}f(x)=f(x)$ almost everywhere.
\end{theorem}

\noindent Recently, Bourgain \cite{jB16} gave examples showing that such convergence can fail for any $s < 1/3$, and so Theorem \ref{Thm:PointConv} is sharp up to the endpoint.

This problem originates from  Carleson \cite{lC}, who proved convergence
for $s \geq 1/4$ when $n=1$. Dahlberg and Kenig \cite{DK} showed that the convergence does not hold for $s<1/4$
in any dimension. Sj\"{o}lin \cite{pS} and Vega \cite{lV} proved independently the convergence for $s>1/2$ in all dimensions.
The sufficient condition for pointwise convergence was improved by Bourgain \cite{jB}, Moyua-Vargas-Vega \cite{MVV},
and Tao-Vargas \cite{TV}.  The best known sufficient condition in dimension $n=2$ was $s>3/8$ , due to Lee \cite{sL}
using Tao-Wolff's bilinear restriction method.  In general dimension $n\geq 2$, Bourgain \cite{jB12} showed the convergence for $s>1/2-1/(4n)$, 
 using multilinear methods.  When $n=2$, this approach gives a different proof of Lee's result for $s > 3/8$.  
 
 For many years, it had seemed plausible that convergence actually 
 holds for $s > 1/4$ in every dimension.  Only in 2012, Bourgain \cite{jB12} gave a counterexample showing that this is false in sufficiently high dimensions.  
Improved counterexamples were given by Luc\'a-Rogers \cite{LR} \cite{LR17} and Demeter-Guo \cite{DG}.  Very recently, in \cite{jB16}, Bourgain gave counterexamples showing
 that convergence can fail if  $s < \frac{n}{2(n+1)}$.  In particular, for $n=2$, convergence can fail if $s < 1/3$.
\\

We will follow the standard approach by bounding the associated maximal function.  We use $B^n(c, r)$ to represent a ball centered at $c$ with radius 
$r$ in $\mathbb R^n$,  and use $\hichi_{E}$ to denote the characteristic function of any measurable set $E$.
For brevity, $B(c,r)$ represents $B^2(c,r)$, a ball in $\mathbb R^2$. 

\begin{theorem} \label{Thm:SchMaxBound} 
For any $s > 1/3$, the following bound holds: for any function $f \in H^s(\mathbb{R}^2)$,
\begin{equation} \label{maxLp}
\left\| \sup_{0 < t \le 1} | e^{it \Delta} f| \right\|_{L^3(B(0,1))}  \le C_s \| f \|_{H^s(\mathbb{R}^2)}.
\end{equation}
\end{theorem}

If the support of $\wh f$ lies in
$A(R)=\{\xi\in \ZR^2:|\xi|\sim R\}$, then Theorem \ref{Thm:SchMaxBound} boils down to the bound

\begin{equation} \label{maxL3R}
\left\| \sup_{0 < t \le 1} | e^{it \Delta} f| \right\|_{L^3(B(0,1))}  \le C_\eps R^{1/3 + \eps} \| f \|_{L^2}.
\end{equation}

\noindent After parabolic rescaling, this bound reduces to the following estimate for functions $f$ with $\wh f$ supported in $A(1)$.

\begin{theorem}\label{thm-3}
For any $\epsilon >0$, there exists a constant $C_\epsilon$ such that
\begin{equation}\label{max4}
 \big\|\underset{0<t\leq R}{\text{sup}}|e^{it\Delta}f|\big\|_{L^3(B(0,R))} \leq
C_\epsilon R^{\epsilon} \|f\|_2
\end{equation}
 holds for all $R\geq 1$ and all $f$ with ${\rm supp}\widehat{f}\subset A(1)=\{\xi\in \ZR^2:|\xi|\sim 1\}$.
\end{theorem}

%\begin{remark}
%\end{remark}

Here is an outline of the proof of Theorem \ref{thm-3}.
The proof uses polynomial partitioning.  This technique was introduced by Nets Katz and the second author in \cite{lGnK}, where it was applied to incidence geometry.  In \cite{lG} and \cite{lG16}, the second author applied this technique to restriction estimates in Fourier analysis.  Polynomial partitioning is a divide and conquer technique.  We begin by finding a polynomial whose zero set divides some object of interest into equal pieces.  For instance, in \cite{lGnK}, it was proven that for any finite volume set $E \subset \ZR^3$ and any degree $D \ge 1$, there is a polynomial $P$ of degree at most $D$ so that $\ZR^3 \setminus Z(P)$ is a union of $\sim D^3$ disjoint open sets $O_i$, and the volumes $| O_i \cap E|$ are all equal.  Hence for any $i$, $|E| \lesssim D^3 | O_i \cap E|$.  In our paper, we choose the polynomial $P$ to behave well with respect to the $L^p_x L^q_t$ norm of $e^{i t \Delta} f$.  For any $p \le q < \infty$ and any degree $D \ge 1$, we show that there is a polynomial $P$ of degree at most $D$ so that $\ZR^3 \setminus Z(P)$ is a union of $\sim D^3$ disjoint open sets $O_i$, and for any $i$,

\begin{equation} \label{equiLpLq}
\| e^{i t \Delta} f \|_{L^p_x L^q_t(B(0,R) \times [0,R])}^p \lesssim D^3 \| \hichi_{O_i} e^{i t \Delta} f \|_{L^p_x L^q_t (B(0,R) \times [0,R])}^p.
\end{equation}

\noindent (To prove Theorem \ref{thm-3}, we will use $q$ finite but very large and $p$ close to 3.  The degree $D$ will be a tiny power of $R$, so $D$ is large compared to 1, but very small compared to $R$.)

Breaking spacetime into cells $O_i$ is useful because of the way it interacts with the wave packet decomposition of $e^{i t \Delta} f$, which we now recall.  We decompose $f$ into pieces that are localized in both physical space and frequency space.  We tile the physical space $B(0,R)$ with $R^{1/2}$-cubes $\nu$, and we tile the frequency space $B(0,1)$ with $R^{-1/2}$-cubes $\theta$.  Then we decompose $f$ as $f = \sum_{\theta, \nu} f_{\theta, \nu}$, where $f_{\theta,\nu}$ is essentially supported on $\nu$ in physical space and essentially supported on $\theta$ in frequency space.  Each function $e^{i t \Delta} f_{\theta, \nu}$ is called a wave packet.  The restriction of $e^{i t \Delta} f_{\theta, \nu}$ to the domain $B(0,R) \times [0,R]$ is essentially supported on a tube $T_{\theta, \nu}$ of radius $R^{1/2}$ and length $R$.  This tube intersects the time slice $\{ t = 0 \}$ at $\nu$, and the direction of the tube depends on $\theta$.

A key fact in the applications of polynomial partitioning in combinatorics is that a line can enter at most $D + 1$ of the cells $O_i$.  To see this, we note that the polynomial $P$ can vanish at most $D$ times along a line, unless it vanishes on the whole line, and so a line can cross $Z(P)$ at most $D$ times.  A wave packet $e^{i t \Delta} f_{\theta, \nu}$ is supported on a tube $T_{\theta, \nu}$ of radius $R^{1/2}$.  This tube can potentially enter many or even all the cells $O_i$, but it cannot penetrate deeply into very many cells.  We define $W$ to be the $R^{1/2}$-neighborhood of $Z(P)$ in $B(0,R) \times [0,R]$, and we define $O_i'$ to be $O_i \setminus W$.  Now the central line of $T_{\theta, \nu}$ can enter at most $D+1$ of the original cells $O_i$, and so the tube $T_{\theta, \nu}$ can enter at most $D+1$ of the smaller cells $O_i'$.  In other words, each wave packet $e^{i t \Delta} f_{\theta, \nu}$ is essentially supported on the union of $W$ and $D+1$ cells $O_i'$.

We can use induction to study $e^{i t \Delta} f$ on each smaller cell $O_i'$.  To study $e^{i t \Delta} f$ on a cell $O_i'$, we only need to take account of those wave packets that intersect $O_i'$.  Therefore, we define $f_i$ to be the sum of $f_{\theta, \nu}$ over those pairs $(\theta, \nu)$ for which $T_{\theta, \nu}$ enters $O_i'$.  On the cell $O_i'$, $e^{i t \Delta} f$ is essentially equal to $e^{i t \Delta} f_i$.  We can control the $L^2$ norms of the $f_i$ by using the fact that $f_{\theta, \nu}$ are (approximately) orthogonal and the fact that each tube $T_{\theta, \nu}$ enters $\lesssim D$ smaller cells $O_i'$.  In particular, we will prove that

$$ \sum_i \| f_i \|_2^2 \lesssim D \| f \|_2^2. $$

\noindent We can now use induction to control $e^{i t \Delta} f$ on each cell $O_i'$.  In this way, we get good control of the contribution to $\| e^{i t \Delta} f \|_{L^p_x L^q_t (B(0,R) \times [0,R])}$ coming from the union of all smaller cells $O_i'$.  It remains to control the contribution coming from $W$.

The most difficult scenario is the following: $e^{i t \Delta} f$ is a sum of wave packets $e^{i t \Delta} f_{\theta, \nu}$ for which the tubes $T_{\theta, \nu}$ are all contained in $W$.  The polynomial partitioning method allows us to reduce the original problem to this special scenario.  This scenario indeed occurs in Bourgain's example in \cite{jB16}.  Let us take a moment to describe this example.

In the example from \cite{jB16}, the zero set $Z(P)$ can be taken to be a plane $t = x_1$.  The set $W$ is a planar slab of thickness $R^{1/2}$.  The solution $e^{i t \Delta} f$ is essentially supported in $W$.  On the plane $t = x_1$, $e^{i t \Delta} f$ is a solution of the Schr\"odinger equation in 1 + 1 dimensions.  In other words, we can choose coordinates $(y,s)$ on this plane and an initial data $g$ so that $e^{i s \Delta} g$ is essentially equal to $e^{i t \Delta} f$ on the plane.  Also, $|e^{i t \Delta} f(x_1, x_2)|$ is approximately constant as we vary $x_1$ within the slab $W$.  The initial data is chosen so that $|e^{i s \Delta} g(y)|$ is large on a set $X$ of $\sim R^{3/2}$ unit squares in $[0,R] \times [0, R]$.  It follows that $|e^{i t \Delta} f(x)|$ is large on a set of $\sim R^{3/2}$ 3-dimensional rectangles of dimensions $R^{1/2} \times 1 \times 1$ in $B(0,R) \times [0,R]$.  Moreover, the projections of these rectangles are roughly disjoint, and so they cover a positive proportion of $B(0,R)$.  Therefore $\sup_{0 < t < R} | e^{i t \Delta} f(x)|$ is large on a positive proportion of $B(0,R)$.

In this construction, the set $X$ needs to be fairly sparse because the projections of the $R^{1/2} \times 1 \times 1$ rectangles need to be disjoint in $B(0,R)$.  In particular, there can be at most $R^{1/2}$ unit squares of $X$ in any $R^{1/2}$-ball in $[0,R] \times [0,R]$.  In the example of \cite{jB16},
$| e^{i s \Delta} g| \sim R^{-5/12} \| g \|_{L^2([0,R])}$ on the set $X$.  During our proof, we will need to show that this quantity $R^{-5/12} \| g \|_{L^2}$ could not be any larger.  In rough terms, we need to show that a solution $e^{i s \Delta} g$ cannot focus too much on a set $X$ which is sparse and spread out.

We will prove such bounds using the $l^2$ decoupling theorem of Bourgain and Demeter \cite{BD}.  We think of these bounds as refinements of the Strichartz inequality.  Here is one such estimate:

\begin{theorem} \label{refstrichintro} Suppose that $g: \ZR \rightarrow \mathbb{C}$ has frequency supported in $B^1(0,1)$.  Suppose that $Q_1, Q_2, ...$ are lattice  $R^{1/2}$-cubes in $[0,R]^2$, so that

$$ \| e^{i t \Delta} g \|_{L^6(Q_j)} \textrm{ is essentially constant in $j$}. $$

\noindent Suppose that these cubes are arranged in horizontal strips of the form $\ZR \times \{t_0, t_0 + R^{1/2} \}$, and that each strip contains $\sim \sigma$ cubes $Q_j$.  Let $Y$ denote $\bigcup_j Q_j$.  Then
for any $\eps > 0$, 

$$ \| e^{i t \Delta} g \|_{L^6(Y)} \le C_\eps R^\eps \sigma^{-1/3} \| g \|_{L^2}. $$
\end{theorem}

 \begin{figure}  [ht]
\centering
\includegraphics[scale=.6]{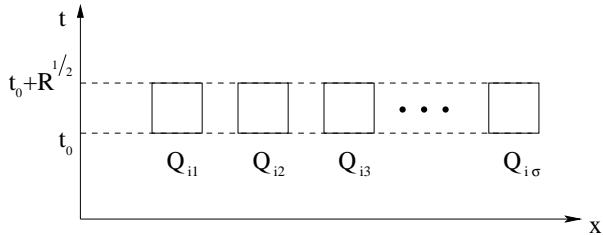}
\caption{\small{ $\sim \sigma$ many cubes in a horizontal strip }}
\label{figure:Fig-strip}
\end{figure}

The Strichartz inequality says that $\| e^{i t \Delta} g\|_{L^6([0,R]^2)} \lesssim \| g \|_{L^2}$.  Theorem \ref{refstrichintro} says that we get a stronger estimate when the solution $e^{i t \Delta} g$ is spread out in space.  To get a sense of what the theorem says, consider the following example.  Suppose that $e^{i t \Delta} g$ is a sum of $\sigma$ wave packets supported on disjoint $R^{1/2} \times R$ rectangles.  We can take $Y$ to be the union of these rectangles.  By scaling, we can suppose that $| e^{i t \Delta} g | \sim 1$ on these $\sigma$ rectangles and negligibly small elsewhere, and then a direct calculation shows that $\| e^{i t \Delta} g \|_{L^6(Y)} \sim \| e^{i t \Delta} g \|_{L^6([0,R]^2)} \sim \sigma^{-1/3} \| g \|_{L^2([0,R])}$.
So Theorem \ref{refstrichintro} roughly says that if $e^{i t \Delta} g$ is ``as spread out as'' $\sigma$ disjoint wave packets, then its $L^6$ norm cannot be much bigger than the $L^6$ norm of $\sigma$ disjoint wave packets.  

This theorem helps us to control the size of $e^{i t \Delta} g$ on a sparse, spread out set $X$ as above.  Suppose that the function $e^{i t \Delta} g$ is evenly spread out on $[0,R]^2$ in the sense that $\| e^{ i t \Delta} g \|_{L^6(Q)}$ is roughly constant among all $R^{1/2}$-boxes $Q \subset [0,R]^2$.  
 In this case, we can take $\sigma = R^{1/2}$ in Theorem \ref{refstrichintro}, which gives 

$$ \| e^{i t \Delta} g \|_{L^6([0,R]^2)} \lesssim R^{-1/6 + \eps} \| g \|_{L^2}. $$

\noindent In the example from \cite{jB16}, $X$ contains $\sim R^{1/2}$ unit squares in each $R^{1/2}$-box of $[0,R]^2$, and each of these boxes indeed has a roughly equal value of $\| e^{i t \Delta} g \|_{L^6(Q)}$.  If $| \eit g | \sim H$ on the set $X$, then Theorem \ref{refstrichintro} gives

$$ H |X|^{1/6} \lesssim  \| e^{i t \Delta} g \|_{L^6([0,R]^2)} \lesssim R^{-1/6 + \eps} \| g \|_{L^2}. $$

\noindent Since $|X| \sim R^{3/2}$, we get the bound $H \lesssim R^{-5/12 + \eps} \| g \|_{L^2}$.  This upper bound matches the behavior of the example from \cite{jB16} up to a factor $R^\eps$.  

Theorem \ref{refstrichintro} lets us deal with the case that $Z(P)$ is a plane.  We need to deal with the more general case that $Z(P)$ is a possibly curved surface of degree at most $D$.  We prove a more general version of Theorem \ref{refstrichintro}, Theorem \ref{refstrich}, which covers the case of wave packets concentrated into a curved surface.

\begin{acknowledgement} The second author is supported by a Simons Investigator grant.
\end{acknowledgement}

\section{Main inductive theorem}

Here we state a slightly more complicated theorem which will imply all the theorems in the introduction.  Our proof uses induction, and we need the slightly more complicated formulation to make all the inductions work.  First of all, the polynomial partitioning involves a topological argument, and the topological argument does not work well with the $\sup$ appearing in our maximal function.  Therefore, we replace the norm $L^p_x L^\infty_t$ with the norm $L^p_x L^q_t$ for $q$ very large.  Another technical issue has to do with parabolic rescaling.  Suppose that $\wh f$ is supported in a smaller ball $B(\xi_0, M^{-1}) \subset B(0,1)$.  In this situation, one can often apply parabolic rescaling to reduce the problem at hand to a problem on a smaller ball in physical space.  However, the change of coordinates in such a parabolic rescaling does not interact well with mixed norms of the form $L^p_x L^q_t$.  Therefore, we instead do induction on the size of the ball $B(\xi_0, M^{-1})$, proving slightly stronger bounds when the ball is small.  Taking account of these small issues, we formulate our result in the following way:

\begin{theorem}\label{thm-4}
For $p > 3$, for any $\epsilon >0$, there exists a constant $C_\epsilon$ such that for any $q>1/\epsilon^4$,
\begin{equation}\label{max5}
 \big\|e^{it\Delta}f\big\|_{L^p_x L^q_t (B(0,R)\times [0,R])} \leq
C_{p, \epsilon} M^{-\epsilon^2}R^{\epsilon} \|f\|_2
\end{equation}
 holds for all $R\geq 1$, any $\xi_0 \in B^2(0,1)$, any $ M\geq 1$ and all $f$ with ${\rm supp}\widehat{f}\subset B^2(\xi_0, M^{-1})$.
\end{theorem}

Let us quickly explain how Theorem \ref{thm-4} implies the theorems in the introduction.  We note that by the dominated
convergence theorem we have
\[
  \big\|\underset{0<t\leq R}{\text{sup}}|e^{it\Delta}f|\big\|_{L^p(B(0,R))} =
  \lim_{q\to \infty}\big\|e^{it\Delta}f\big\|_{L^p_x L^q_t (B(0,R)\times [0,R])},
\]
 for any $L^2$-function $f$ with compact Fourier support or any Schwartz function $f$.  Therefore, Theorem \ref{thm-4} implies that for any $R \ge 1$ and any $f$ with the support of $\wh f \subset B(0,1)$, and for any $p > 3$, we have

\begin{equation} \label{checkforp>3}
 \big\|e^{it\Delta}f\big\|_{L^p_x L^\infty_t (B(0,R)\times [0,R])} \leq
C_{p, \epsilon} R^{\epsilon} \|f\|_2
\end{equation}

\noindent So far we assume $p>3$.  But it is straightforward to prove a bound of the form

$$\| e^{it\Delta} f\|_{L^2_x L^\infty_t (B(0,R) \times [0,R])} \le R^{O(1)} \| f \|_2.$$

\noindent Combining these bounds using H\"older's inequality, we see that Equation (\ref{checkforp>3}) holds for $p=3$ as well.  This establishes Theorem \ref{thm-3}.

We write $A \lessapprox B$ if $A \le C_\eps R^\eps B$ for any $\eps > 0$.  Suppose now that $\wh g$ is supported in $A(R)$.  To prove Theorem \ref{Thm:SchMaxBound}, we want to show that

\begin{equation} \label{maxL3R'}
\left\| \sup_{0 < t \le 1} | e^{it \Delta} g| \right\|_{L^3(B(0,1))}  \lessapprox R^{1/3} \| g \|_{L^2}.
\end{equation}

After parabolic rescaling, we are led to a function $f$ with $\wh f$ supported in $A(1)$, and we need to show the bound

\begin{equation*} \label{maxL3R''}
\left\| \sup_{0 < t \le R^2} | e^{it \Delta} f| \right\|_{L^3(B(0,R))}  \lessapprox \| f \|_{L^2}.
\end{equation*}

\noindent But applying Theorem \ref{thm-3} with $R^2$ in place of $R$ gives:

\begin{equation*} \label{maxL3R'''}
\left\| \sup_{0 < t \le R^2} | e^{it \Delta} f| \right\|_{L^3(B(0,R^2))}  \lessapprox \| f \|_{L^2}.
\end{equation*}

\noindent This implies Equation (\ref{maxL3R'}).  Now, given $s > 1/3$ and $f \in H^s(\mathbb{R}^2)$, we decompose $f$ in a Littlewood-Paley decomposition: $f = \sum_{k \ge 0} f_k$ where $\wh f_0$ is supported in $B(0,1)$ and $\wh f_k$ is supported in $A(2^k)$ for $k \ge 1$.  We have $\| f_k \|_{L^2} \lesssim 2^{-ks} \| f \|_{H^s}$.  Applying (\ref{maxL3R'}) to each $f_k$ and using the triangle inequality, we get Theorem \ref{Thm:SchMaxBound}.

Theorem \ref{Thm:SchMaxBound} implies Theorem \ref{Thm:PointConv} by a standard smooth approximation argument, which we briefly recall.  If $f$ is Schwartz, 
then it is well-known that $e^{i t \Delta} f(x) \rightarrow f(x)$ uniformly in $x$.  Schwartz functions are dense in $H^s$, and so we can write $f = g + h$ 
where $g$ is Schwartz and $\| h \|_{H^s} < \eps^{100}$.  Since $g$ is Schwartz, we can find a time $t_\eps > 0$ so that $| e^{i t \Delta} g(x) - g(x) | < \eps$ 
for all $x$ and all $0 \le t \le t_\eps$.  On the other hand, by the maximal estimate in Theorem \ref{Thm:SchMaxBound}, $|e^{i t \Delta} h(x)| < \eps$ for all $0 \le t \le 1$ and all $x$ in $B(0,1) \setminus X_\eps$, where $|X_\eps| < \eps$.  Taking a sequence of $\eps \rightarrow 0$ exponentially fast, and doing a little measure theory, it follows that $e^{i t \Delta} f(x) \rightarrow f(x)$ for almost every $x \in B(0,1)$.  The same applies to any other ball, and we see that $e^{i t \Delta} f(x) \rightarrow f(x)$ for almost every $x \in \ZR^2$.

We also remark that the local bound \eqref{maxLp} from Theorem \ref{Thm:SchMaxBound} can be used to derive immediately a global estimate in $L^{3}(\mathbb R^2)$ 
for the maximal function $\sup_{0<t\leq 1}|e^{it\Delta }f|$, following from Theorem 10 in \cite{Rogers}.
% provided  $f\in H^{\frac13+}$.
We are indebted to K. Rogers for pointing this out to us.

In the rest of the paper, we prove Theorem \ref{thm-4}.  In Section \ref{secpolypart}, we review polynomial partitioning, and in Section \ref{WPD}, we 
review wave packet decomposition.  Then we begin the proof of Theorem \ref{thm-4} in Section \ref{MP-C}.

\section{Polynomial Partitioning} \label{secpolypart}
\setcounter{equation}0

First we state a variation of the ham-sandwich theorem, which introduces a polynomial $P$ in the polynomial ring $\mathbb {R}[x,t]$ such  that
the variety $Z(P)=\{(x,t)\in \mathbb{R}^n \times \mathbb{R}: P(x,t)=0\}$
 bisects  every member in a collection of some quantities. It relies on Borsuk-Ulam Theorem,  which asserts that
 {\it if $F:\mathbb{S}^N \xrightarrow{} \mathbb{R}^N$ is a continuous function,  where $\mathbb{S}^N$ is the $N$-dimensional unit sphere, then there exists a point $v\in \mathbb{S}^N$ with $F(v)=F(-v)$.}

\begin{lemma} \label{Thm:Sandwich}
  If $W_1, W_2, \cdots, W_N \in L^1_x L^r_t(\mathbb{R}^n\times \mathbb{R}), 1\leq r<\infty$, then there exists a non-zero polynomial $P$ on $\mathbb{R}^n \times \mathbb{R}$ of degree $\leq c_nN^{1/(n+1)}$ such that for each $W_j$,
  \[
    \big\|\hichi_{\{P>0\}}W_j\big\|_{L^1_x L^r_t(\mathbb{R}^n\times \mathbb{R})}=\big\|\hichi_{\{P<0\}}W_j\big\|_{L^1_x L^r_t(\mathbb{R}^n\times \mathbb{R})}.
  \]
\end{lemma}

\begin{proof}
  Let $V$ be the vector space of polynomials on $\mathbb{R}^n \times \mathbb{R}$ of degree at most $D$, then
\[
  \text{Dim} V = \binom{D+n+1}{n+1} \sim_n D^{n+1}.
\]
So we can choose $D \sim N^{1/(n+1)}$ such that $\text{Dim} V\geq N+1$, and without loss of generality we can assume $\text{Dim} V = N+1$ and identify $V$ with $\mathbb{R}^{N+1}$. We define a function $G$ as follows:
\begin{align}
  \mathbb{S}^N \subseteq V \backslash \{0\} & \xlongrightarrow{G} \mathbb{R}^N \notag\\
  P   & \mapsto \{G_j(P)\}_{j=1}^{N} \,,\notag
\end{align}
where
\[
  G_j(P):= \big\|\hichi_{\{P>0\}}W_j\big\|_{L^1_x L^r_t(\mathbb{R}^n\times \mathbb{R})}-\big\|\hichi_{\{P<0\}}W_j\big\|_{L^1_x L^r_t(\mathbb{R}^n\times \mathbb{R})} ,
\]
it is obvious that $G(-P)=-G(P).$
Assume that the function $G$ is continuous, then Borsuk-Ulam Theorem tells us that there exists $P\in \mathbb{S}^N \subseteq V
\backslash \{0\}$ with $G(P)=G(-P)$, hence $G(P)=0$, and $P$ obeys the conclusion of Lemma \ref{Thm:Sandwich}.
It remains to check the continuity of the functions $G_j$ on $V \backslash \{0\}$.

 Suppose that $P_k \to P$ in $V \backslash \{0\}$. Note that
\[
  |G_j(P_k)-G_j(P)| \leq 2 \big\|\hichi_{\{P_kP\leq 0\}}W_j\big\|_{L^1_x L^r_t(\mathbb{R}^n\times \mathbb{R})}\,,
\]
while $P_k \to P$ implies that
\[
  \bigcap_{k_0}\bigcup_{k\geq k_0}\{(x,t): P_k(x,t)\cdot P(x,t)\leq0\} \subseteq P^{-1}(0).
\]
By the dominated convergence theorem,
\[
  \lim_{k_0\to\infty} \big\|\hichi_{\cup_{k\geq k_0}\{P_k P\leq0\}}W_j\big\|_{L^1_x L^r_t(\mathbb{R}^n\times \mathbb{R})}
  = \big\|\hichi_{\{P^{-1}(0)\}}W_j\big\|_{L^1_x L^r_t(\mathbb{R}^n\times \mathbb{R})}=0.
\]
This proves that $\lim_{k\to\infty}|G_j(P_k)-G_j(P)| =0$, showing that $G_j$ is continuous on $V\backslash\{0\}.$
\end{proof}

By applying Lemma \ref{Thm:Sandwich} repeatedly, we get the following polynomial partitioning result:

\begin{theorem} \label{Thm:PolynPart}
  If $W \in L^1_x L^r_t(\mathbb{R}^n\times \mathbb{R})\backslash\{0\},1\leq r<\infty$, then for each $D$ there exists a non-zero polynomial $P$ of degree at most D such that $(\mathbb{R}^n\times\mathbb{R})\backslash Z(P)$ is a union of $\sim_n D^{n+1}$ disjoint open sets $O_i$ and for each $i$ we have
  \[
  \big\|W\big\|_{L^1_x L^r_t(\mathbb{R}^n\times \mathbb{R})} \leq c_n D^{n+1}\big\|\hichi_{O_i}W\big\|_{L^1_x L^r_t(\mathbb{R}^n\times \mathbb{R})}.
  \]
\end{theorem}

\begin{proof}
  By Lemma \ref{Thm:Sandwich}, we obtain a polynomial $P_1$ of degree $\lesssim 1$ such that
  \[
    \big\|\hichi_{\{P_1>0\}}W\big\|_{L^1_x L^r_t(\mathbb{R}^n\times \mathbb{R})}=\big\|\hichi_{\{P_1<0\}}W\big\|_{L^1_x L^r_t(\mathbb{R}^n\times \mathbb{R})}.
  \]
Next, we let $W_+:=\hichi_{\{P_1>0\}}W$ and $W_-:=\hichi_{\{P_1<0\}}W$, and by Lemma \ref{Thm:Sandwich} again we obtain a polynomial $P_2$ of degree $\lesssim 2^{1/(n+1)}$ such that
\[
  \big\|\hichi_{\{P_2>0\}}W_j\big\|_{L^1_x L^r_t(\mathbb{R}^n\times \mathbb{R})}=\big\|\hichi_{\{P_2<0\}}W_j\big\|_{L^1_x L^r_t(\mathbb{R}^n\times \mathbb{R})},
\]
for $j=+,-$. Continuing inductively, we construct polynomials $P_1, P_2, \cdots, P_s$. Let $P:=\prod_{k=1}^s P_k$. The sign conditions of the polynomials cut $(\mathbb{R}^n\times\mathbb{R})\backslash Z(P)$ into $2^s$ cells $O_i$, and by construction and triangle inequality we have that, for each $i$,
\[
  \big\|W\big\|_{L^1_x L^r_t(\mathbb{R}^n\times \mathbb{R})} \leq 2^s \big\|\hichi_{O_i}W\big\|_{L^1_x L^r_t(\mathbb{R}^n\times \mathbb{R})}.
\]
 By construction, $\text{deg}\, P_k \lesssim 2^{(k-1)/(n+1)}$, therefore $\text{deg}\, P \leq c_n2^{s/(n+1)}$. We can choose $s$ such that $c_n2^{s/(n+1)} \in [D/2,D]$, then $\text{deg}\, P \leq D$ and the number of cells $2^s \sim_{n} D^{n+1}$.
\end{proof}

\begin{definition}
  We say that a polynomial $P$ is non-singular if $\nabla P(z) \neq 0$ for each point $z$ in $Z(P)$.
\end{definition}

It is well-known that non-singular polynomials are dense in the space of all polynomials, cf. Lemma 1.5 in \cite{lG}.
Following from the density of non-singular polynomials and the proof of Theorem \ref{Thm:PolynPart}, we can assume that the polynomial in the partitioning theorem enjoys nice geometric properties.

\begin{theorem} \label{Thm:PolynPart2}
  If $W \in L^1_x L^r_t(\mathbb{R}^n\times \mathbb{R})\backslash\{0\}, 1\leq r <\infty$, then for each $D$ there exists a non-zero polynomial $P$ of degree at most D such that $(\mathbb{R}^n\times\mathbb{R})\backslash Z(P)$ is a union of $\sim_n D^{n+1}$ disjoint open sets $O_i$ and for each $i$ we have
  \[
  \big\|W\big\|_{L^1_x L^r_t(\mathbb{R}^n\times \mathbb{R})}\leq c_nD^{n+1}\big\|\hichi_{O_i}W\big\|_{L^1_x L^r_t(\mathbb{R}^n\times \mathbb{R})}.
  \]
  Moreover, the polynomial $P$ is a product of distinct non-singular polynomials.
\end{theorem}

\section{Wave Packet Decomposition}\label{WPD}
\setcounter{equation}0

We focus on the dimension $n=2$ in the rest of the paper.

A (dyadic) rectangle in $\mathbb R^{2}$ is a product of (dyadic) intervals with respect to given coordinate axes of
$\mathbb R^{2}$. A rectangle $\theta=\prod_{j=1}^{2}{\theta_j}$ in frequency space and a rectangle $\nu=\prod_{j=1}^{2}\nu_j$ in physical space are said to
be dual if $|\theta_j||\nu_j|=1$ for $j=1, 2$. We say that $(\theta ,\nu)$ is a tile if it is a pair of \emph{dual} (dyadic) rectangles.  The dyadic condition is not essential in our decomposition.

Let $\varphi$ be a Schwartz function from $\mathbb R $ to $\mathbb R$ whose Fourier transform is non-negative, supported
in a small interval, of radius $\kappa$ ($\kappa$ is a fixed small constant), about the origin in $\mathbb R$, and identically
$1$ on another smaller interval around the origin. For a (dyadic) rectangular box $\theta=\prod_{j=1}^2 \theta_j$,
set
\begin{equation}
 \widehat{\varphi_\theta}(\xi_1, \xi_2) = \prod_{j=1}^{2} \frac{1}{|\theta_j|^{1/2}}\widehat\varphi\left( \frac{\xi_j-c(\theta_j)}{|\theta_j|}\right)\,.
\end{equation}

Here $c(\theta_j)$ is the center of the interval $\theta_j$ and hence $c(\theta)=\left(c(\theta_1),  c(\theta_{2})\right)$
 is the center of the rectangle $\theta$.  We also note that $\| \varphi_\theta \|_{L^2} \sim 1$.
We let $c(\nu)$ denote the center of $\nu$.  For a tile $(\theta, \nu)$ and $x\in \mathbb R^{2}$, we define
\begin{equation}
\widehat{\varphi_{\theta, \nu}}(\xi)= e^{2\pi ic(\nu)\cdot \xi}\widehat{\varphi_\theta}(\xi).
\end{equation}
We say that two tiles $(\theta, \nu)$ and $(\theta', \nu')$ have the same dimensions if $|\theta_j| = |\theta'_j|$ for all $j$, which then implies that $|\nu_j| = |\nu_j'|$ for all $j$.
Let $\bT$ be a collection of all tiles with fixed dimensions and coordinate axes.
Then for any Schwartz function $f$ from $\mathbb R^{2}$ to $ \mathbb R$, we have the following
representation
\begin{equation}\label{rep0}
 f(x) = c_\kappa \sum_{(\theta, \nu)\in \bT} f_{\theta,\nu}
 := c_\kappa \sum_{(\theta, \nu)\in \bT}\langle f, \varphi_{\theta,\nu}\rangle \varphi_{\theta,\nu}(x)\,,
\end{equation}
where $c_\kappa$ is an absolute constant.
This representation can be proved directly (see \cite{LL}) or
by employing inductively the one-dimensional result in \cite{LT1}. \\

We will only use tiles $(\theta, \nu)$ where $\theta$ is an $R^{-\frac12}$-cube in frequency space and $\nu$ is an $R^{\frac12}$-cube in physical space.  Indeed, let $\theta$ be an $R^{-\frac12}$-cube (or ball)
in $B(0,1)\subset \ZR^2$. Let $\bT_\theta$ be a collection of all tiles $(\theta',\nu)$
such that $\nu$'s are $R^{\frac12}$-cubes and $\theta'=\theta$. Then for any Schwartz function $f$ with ${\rm supp}\wh{f}\subset B(0,1)$, we have
\begin{equation}\label{rep1}
 f(x) = c_\kappa\sum_{\theta}\sum_{(\theta',\nu)\in\bT_\theta}\langle f, \varphi_{\theta',\nu}\rangle \varphi_{\theta',\nu}(x)\,.
\end{equation}
Here $\theta$'s range over all possible cubes in ${\rm supp}\widehat{f}$. We use
$\bT$ to denote $\cup_\theta \bT_\theta$. It is clear that
\begin{equation}
 \sum_{(\theta,\nu)\in \bT}\big|\langle f,\varphi_{\theta,\nu}\rangle\big|^2\sim \|f\|_2^2\,.
\end{equation}
We set
\begin{equation}
 \psi_{\theta,\nu}(x, t) = e^{it\Delta}\varphi_{\theta,\nu}(x)\,.
\end{equation}
From (\ref{rep1}), we end up with the following representation for $e^{it\Delta}f$ :
\begin{equation}\label{rep2}
 e^{it\Delta}f(x) =c_\kappa \sum_{(\theta,\nu)\in \bT} e^{it\Delta}f_{\theta,\nu}(x)= c_\kappa \sum_{(\theta,\nu)\in \bT} \langle f, \varphi_{\theta,\nu}\rangle\psi_{\theta,\nu}(x, t)\,.
\end{equation}
We shall analyze the localization of $\psi_{\theta,\nu}$ in the physical and frequency space.  

On the domain $B(0,R) \times [0,R]$, the function $\psi_{\theta, \nu}$ is essentially supported on a tube $T_{\theta, \nu}$ defined as follows. Let
\begin{equation}\label{defofTs}
  T_{\theta,\nu} := \{(x,t)\in\mathbb{R}^2\times\mathbb{R}\,:\,0\leq t\leq R, |x-c(\nu)+2tc(\theta)|\leq R^{1/2+\delta}\}\, ,
\end{equation}
where $\delta=\epsilon^2$ is a small positive parameter. We see that $T_{\theta,\nu}$ is a tube of length $R$, of radius $R^{1/2+\delta}$, in the
direction $G_0(\theta)=(-2c(\theta),1)$, and intersecting $\{t=0\}$ at an $R^{1/2+\delta}$-ball centered at $c(\nu)$.
In order to see this, let $\psi$ be a Schwartz function with Fourier transform supported in $[-1, 1]$ and $2\psi(t)\geq \hichi_{[0,1]}(t)$.
Here $\hichi_{[0,1]}$ is the characteristic function on $[0,1]$.  On $B(0,R) \times [0, R]$, we have $| \psi_{\theta, \nu}| \le 2 | \psi^*_{\theta, \nu}|$, where
\begin{equation}
 \psi_{\theta,\nu}^*(x,t) =\psi_{\theta,\nu}(x, t) \psi\big(\frac{t}{R}\big)\,.
\end{equation}
From the definitions of $e^{it\Delta}$ and $\psi_{\theta,\nu}$, it is easy to check that, by integration by parts,
$\psi^*_{\theta,\nu}$ is essentially supported in the tube $T_{\theta,\nu}$.
More precisely, we have
\begin{equation}\label{eq:psis}
 |\psi^*_{\theta,\nu}(x,t)|\leq \frac{1}{\sqrt{R}} \hichi^*_{T_{\theta,\nu}}(x,t) \, ,
\end{equation}
where
$\hichi^*_{T_{\theta,\nu}}$ denotes a bump function satisfying that
$\hichi^*_{T_{\theta,\nu}}=1$ on $\{(x,t)\in\mathbb{R}^2\times\mathbb{R}\,:\,0\leq t\leq R, |x-c(\nu)+2tc(\theta)|\leq \sqrt{R}\}$,
and $\hichi^*_{T_{\theta,\nu}}=O(R^{-1000})$ outside $T_{\theta,\nu}$. We can essentially treat $\hichi^*_{T_{\theta,\nu}}$ as $\hichi_{T_{\theta,\nu}}$, the indicator function on the tube $T_{\theta,\nu}$.

On the other hand, the Fourier transform of $\psi_{\theta,\nu}^*$ enjoys
\begin{equation}
  \widehat{\psi^*_{\theta,\nu}}(\xi_1,\xi_2,\xi_3)
 =R\widehat{\varphi_{\theta,\nu}}(\xi_1,\xi_2)\widehat{\psi}\left(\frac{\xi_3-(\xi_1^2+\xi_2^2)}{1/R}\right)\,.
\end{equation}
Hence $\widehat{\psi^*_{\theta,\nu}}$ is supported in the $\frac{1}{R}$-neighborhood of the parabolic cap over $\theta$, that is,
\begin{equation} \label{eq:psishat}
  \text{supp}\,\widehat{\psi^*_{\theta,\nu}} \subseteq \big\{(\xi_1,\xi_2,\xi_3)\,:\,
(\xi_1,\xi_2)\in \theta, |\xi_3-(\xi_1^2+\xi_2^2)|\leq \frac{1}{R} \big\}.
\end{equation}
We denote this $\frac{1}{R}$-neighborhood of the parabolic cap over $\theta$ by $\theta^*$.
In the rest of the paper, we can assume that the function $\psi_{\theta,\nu}$ is essentially localized in
$T_{\theta,\nu}$ in physical space, and $\theta^*$ in frequency space.

\section{Cell contributions}
\label{MP-C}
\setcounter{equation}0

The rest of the paper is devoted to a proof of Theorem \ref{thm-4}, using polynomial partitioning. Recall that the functions $f$ in Theorem \ref{thm-4} are Fourier supported in $B(\xi_0,M^{-1})\subset\mathbb{R}^2$ with arbitrary $\xi_0\in B(0,1)$ and $ M\geq 1$.  Also $p > 3$ and $q > \eps^{-4}$.  The function $f$ can be assumed to be a Schwartz function since the collection of all Schwartz
functions is dense in $L^2$.  We need to prove the bound (\ref{max5}):

\begin{equation*}
 \big\|e^{it\Delta}f\big\|_{L^p_x L^q_t (B(0,R)\times [0,R])} \leq
C_{p, \epsilon} M^{-\epsilon^2}R^{\epsilon} \|f\|_2\,.
\end{equation*}

The proof of Theorem \ref{thm-4} is by induction on the radius $R$ in physical space and the radius $1/M$ in frequency space.  First we cover the bases of the induction.  If $M\geq R^{10}$, then we bound $|e^{it\Delta}f(x)|$ by $M^{-1}\|f\|_2$ and Theorem \ref{thm-4} is trivial. If $R^{1/2-O(\delta)}<M<R^{10}$, then all associated wave packets are in the same direction, and by a direct computation we can bound the left-hand side of \eqref{max5} by $R^{(3-p)/(2p)+O(\delta)}\|f\|_2$, from which Theorem \ref{thm-4} follows immediately. Therefore we can assume that $M\ll \sqrt R$.  We can assume that $R$ is sufficiently large, otherwise Theorem \ref{thm-4} is trivial.  This covers the base of the induction.  Now we turn to the inductive step.  By induction, we can assume that Theorem \ref{thm-4} holds for physical radii less than $R/2$ or for physical radius $R$ and frequency radius less than $\frac{1}{2M}$.

Let $\bB^*_R$ denote the set $B(0, R)\times [0, R]$.

We pick a degree $D=R^{\epsilon^4}$, and apply polynomial partitioning with this degree to the
function $\hichi_{\bB^*_R}|e^{it\Delta}f(x)|^p$. By Theorem \ref{Thm:PolynPart2} with $r=q/p$, there exists a non-zero polynomial
$P$ of degree at most $D$ such that $(\mathbb{R}^2\times\mathbb{R})\backslash Z(P)$ is a union of $\sim D^{3}$
disjoint open sets $O_i$ and for each $i$ we have
\begin{equation} \label{eq:PolynPart}
  \big\|e^{it\Delta} f (x)\big\|^p_{L^p_x L^q_t (\bB^*_R)} \leq cD^3 \big\|\hichi_{O_i}e^{it\Delta} f (x)\big\|^p_{L^p_x L^q_t (\bB^*_R)}.
\end{equation}
  Moreover, the polynomial $P$ is a product of distinct non-singular polynomials.

We define
\begin{equation}
  W:=N_{R^{1/2+\delta}}Z(P)\cap \bB^*_R\,,
\end{equation}
where $\delta=\epsilon^2$ and  $N_{R^{1/2+\delta}}Z(P)$ stands for the $R^{1/2+\delta}$-neighborhood of
the variety $Z(P)$ in $\ZR^3$.
We have the wave packet decomposition for $e^{it\Delta}f$ as in (\ref{rep2}). For each cell $O_i$, we set
\begin{equation}
 O'_i :=\left[O_i\cap\bB^*_R\right] \backslash W
\,\,\,{\rm and}\,\,\,
 \bT_i:=\{(\theta,\nu)\in\bT\,:\,T_{\theta,\nu} \cap O'_i \neq \emptyset\} \,.
\end{equation}
Here $T_{\theta,\nu}$ is the tube associated to each tile $(\theta,\nu)$, as defined in (\ref{defofTs}).
For each function $f$ we define
\begin{equation}
  f_i :=  \sum_{(\theta,\nu)\in\bT_i} f_{\theta,\nu} \,.
\end{equation}
From \eqref{eq:psis}, it follows that on each cell $O_i'$,
\begin{equation}\label{eqf-fi}
 e^{it\Delta}f(x)\sim e^{it\Delta}f_i(x) \,.
\end{equation}
By the fundamental theorem of algebra, we have a simple yet important geometric observation:
\begin{lemma}\label{lem-G0}
  For each tile $(\theta,\nu)\in\bT$, the number of cells $O'_i$ that intersect the tube $T_{\theta,\nu}$ is $\leq D+1$.
\end{lemma}
\begin{proof}
 If $T_{\theta,\nu}$ intersects $O'_i$, then the central line of $T_{\theta,\nu}$ must enter $O_i$. On the other hand, a line can cross the variety $Z(P)$ at most $D$ times, hence can enter at most $D+1$ cells $O_i$.
\end{proof}

By triangle inequality, we dominate $\big\|e^{it\Delta} f (x)\big\|^p_{L^p_x L^q_t (\bB^*_R)}$ by
\begin{equation}\label{cell-W}
\sum_i  \big\|\hichi_{O_i'}e^{it\Delta} f (x)\big\|^p_{L^p_x L^q_t (\bB^*_R)}
+  \big\|\hichi_W e^{it\Delta} f (x)\big\|^p_{L^p_x L^q_t (\bB^*_R)}\,.
\end{equation}
We call the first term in (\ref{cell-W}) the cellular term, and the second the wall term.
Using induction we will see that the desired bound (\ref{max5}) holds unless the wall term makes a significant contribution.  In particular, we will show that (\ref{max5}) holds unless

\begin{equation}\label{Wdom}
\big\| e^{i t \Delta} f \big \|_{L^p_x L^q_t (\bB^*_R)} \lesssim R^{O(\eps^4)} \big\| \hichi_W e^{i t \Delta} f \big \|_{L^p_x L^q_t (\bB^*_R)}.
\end{equation}

Define
\begin{equation} \label{eq:I}
  \mathcal{I} =\bigg\{i\,:\,\big\|e^{it\Delta} f (x)\big\|^p_{L^p_x L^q_t (\bB^*_R)}\leq
10 c D^3 \big\|\hichi_{O'_i}e^{it\Delta} f (x)\big\|^p_{L^p_x L^q_t (\bB^*_R)}\bigg\},
\end{equation}
where $c$ is the constant from \eqref{eq:PolynPart}. By triangle inequality and \eqref{eq:PolynPart}, for each $i\in \mathcal{I}^c$, we have

$$
  \big\|e^{it\Delta} f (x)\big\|^p_{L^p_x L^q_t (\bB^*_R)} \leq \frac{10}{9}cD^3 \big\|\hichi_{O_i\cap W}e^{it\Delta} f (x)\big\|^p_{L^p_x L^q_t (\bB^*_R)} $$

$$\lesssim R^{3 \eps^4} \big\|\hichi_{W}e^{it\Delta} f (x)\big\|^p_{L^p_x L^q_t (\bB^*_R)}.$$

So if $\mathcal{I}^c$ is non-empty, then (\ref{Wdom}) holds.  For the moment, we are considering the case where (\ref{Wdom}) does not hold, and so every index $i$ is in $\mathcal{I}$, and hence $|\mathcal{I} | \sim D^3$.

In addition, by Lemma \ref{lem-G0},
\begin{equation}\label{I-est}
  \sum_{i} \|f_{i}\|_2^2
\lesssim (D+1) \sum_{\theta,\nu}\| f_{\theta,\nu}\|_2^2 \lesssim D\|f\|_2^2\,.
\end{equation}
Henceforth, by pigeonhole principle, there exists $i\in\mathcal{I}$ such that
\begin{equation}\label{eq:special i0}
   \|f_{i}\|_2^2
 \lesssim D^{-2}\|f\|_2^2\,.
\end{equation}

Now we use induction: we apply \eqref{max5} to this special $f_i$ at radius $\frac R 2$.  We can cover $B(0,R) \times [0,R]$ by $O(1)$ cylinders with dimensions $B(0,R/2) \times [0, R/2]$.  Therefore, we get the bound

\begin{align}
  &\big\|e^{it\Delta} f (x)\big\|^p_{L^p_x L^q_t (\bB^*_R)}\lesssim D^3 \big\|\hichi_{O'_i}e^{it\Delta} f (x)\big\|^p_{L^p_x L^q_t (\bB^*_R)} \lesssim  D^{3}\big\|e^{it\Delta} f_i (x)\big\|^p_{L^p_x L^q_t (\bB^*_R)}\notag \\
   \lesssim & D^3 \left[C_{p, \epsilon} M^{-\epsilon^2} R^{\epsilon}\|f_i\|_2\right]^p  \lesssim D^{3-p} \left[C_{p,\epsilon} M^{-\epsilon^2} R^{\epsilon}\|f\|_2\right]^p \,.\notag
\end{align}

Recall that $D=R^{\epsilon^4}$, and we can assume $R$ is very large (compared to $p$).  Since $p>3$ we have $D^{3-p} \ll 1$.  Therefore, we see that induction closes (unless \eqref{Wdom} holds).

It remains to prove the desired bounds when \eqref{Wdom} holds -- when the wall term is almost as big as the whole.
\\

\section{Contribution from the wall: transverse and tangent terms}
\label{wall}
\setcounter{equation}0

From Section \ref{MP-C}, it remains
to estimate the wall contribution, the second term in (\ref{cell-W}).
To deal with the contribution from the wall $W$, we break $B^*_R$ into $\sim R^{3\delta}$ balls $B_j$
 of radius $R^{1-\delta}$.  (Recall from the last section that $\delta$ is defined to be $\eps^2$.)

For any tile $(\theta,\nu) \in\bT$,
we say that $T_{\theta,\nu}$ is \emph{tangent} to the wall $W$ in a given ball $B_j$ if
it satisfies that $T_{\theta,\nu}\cap B_j\cap W \neq \emptyset$ and
\begin{equation}
 \text{Angle}(G_0(\theta),T_z[Z(P)])\leq R^{-1/2+2\delta}
\end{equation}
for any non-singular point $z\in 10T_{\theta,\nu}\cap 2B_j\cap Z(P)$.
Recall that $G_0(\theta)=(-2c(\theta),1)$ is the direction of the tube $T_{\theta,\nu}$. Here $T_z[Z(P)]$ stands for the tangent space to the variety $Z(P)$ at the point $z$, and by a non-singular point we mean a point $z$ in $Z(P)$ with $\nabla P(z) \neq 0.$ Since $P$ is a product of distinct non-singular polynomials, the non-singular points are dense in $Z(P)$.  We note that if $T_{\theta, \nu}$ is tangent to $W$ in $B_j$, then $T_{\theta, \nu} \cap B_j$ is contained in the $R^{1/2 + \delta}$-neighborhood of $Z(P) \cap 2 B_j$.  

We say that $T_{\theta,\nu}$ is \emph{transverse}  to the wall $W$ in the ball $B_j$ if it enjoys
that  $T_{\theta,\nu}\cap B_j\cap W \neq \emptyset$ and
\begin{equation}
 \text{Angle}(G_0(\theta),T_z[Z(P)])> R^{-1/2+2\delta}
\end{equation}
for some non-singular point $z\in 10T_{\theta,\nu}\cap 2B_j\cap Z(P)$.

Let $\bT_{j, {\rm tang}}$ represent the collection of
all tiles $(\theta,\nu)\in \bT$ such that $T_{\theta,\nu}$'s are tangent to the wall $W$ in $B_j$,
and $\bT_{j, {\rm trans}}$ denote the collection of
all tiles $(\theta,\nu)\in \bT$ such that $T_{\theta,\nu}$'s are transverse to the wall $W$ in $B_j$.

We define $f_{j,{\rm tang}}:=\sum_{(\theta,\nu)\in \bT_{j,{\rm tang}}}f_{\theta,\nu}$ and $f_{j,{\rm trans}}:=\sum_{(\theta,\nu)\in \bT_{j,{\rm trans}}}f_{\theta,\nu}$.
Then on $B_j \cap W$, we have
\begin{equation}
e^{it\Delta}f(x)\sim e^{it\Delta}f_{j,\rm tang}(x)+e^{it\Delta}f_{j,\rm trans}(x)\,.
\end{equation}

The following Lemma is about how a tube crosses a variety transversely, which was proved by the second author in \cite{lG}. It says that 
$T_{\theta,\nu}$ crosses the wall $W$ transversely in at most $R^{O(\epsilon^4)}$
many balls $B_j$.

\begin{lemma} \label{L:trans} (Lemma 3.5 in \cite{lG})
  For each tile $(\theta,\nu) \in \bT$, the number of $R^{1-\delta}$-balls $B_j$ for which
 $(\theta,\nu)\in \bT_{j,{\rm trans}}$ is at most $\text{Poly}(D)=R^{O(\epsilon^4)}$.
\end{lemma}

%\begin{lemma} \label{L:tang}
%  $\#\{\omega: \exists\, s\in \bS_{j,{\rm tang}}\,\, \text{with}\,\, \omega_s=\omega\}\leq R^{1/2+O(\delta)}$ for any fixed $j$.
%\end{lemma}

%Lemma \ref{L:tang} indicates that only $O(R^{1/2+O(\delta)})$ many separated directions are tangential to $W$
%in any given $B_j$.  \\

For points $(x,t)\in B_j \cap W$, we could break up $e^{i t \Delta} f(x)$ into a transverse term and a tangent term.  However, when we analyze the tangent contribution in subsequent sections, we will need to use a bilinear structure.  So we do a more refined decomposition: we break $e^{it\Delta}f(x)$ into a linear transverse term and a bilinear tangent term.

We decompose $B(\xi_0,M^{-1})\subset\mathbb{R}^2$, the Fourier support of function $f$, into balls $\tau$ of radius $1/(KM)$. Here $K=K(\epsilon)$ is a large parameter. We write $f=\sum_{\tau}f_\tau$, where $\text{supp}\,\widehat{f_\tau}\subseteq \tau$.

We let $B_\epsilon :=\{(x,t)\in B(0,R)\times [0,R]\,:\,\exists\, \tau\,\, \text{s.t.}\, |e^{it\Delta}f_\tau(x)|>K^{-\epsilon^4}|e^{it\Delta}f(x)|\}$.  We will show by induction on the radius $(1/M)$ in frequency space that
the contribution from $B_\epsilon$ is acceptable. In fact, by the definition of $B_\epsilon$,
\begin{equation}
  \big\|\hichi_{B_\epsilon}e^{it\Delta} f (x)\big\|^p_{L^p_x L^q_t (\bB^*_R)}
  \leq K^{\epsilon^4 p}\sum_\tau \big\|e^{it\Delta} f_\tau (x)\big\|^p_{L^p_x L^q_t (\bB^*_R)} \,. \notag
\end{equation}
By applying \eqref{max5} in Theorem \ref{thm-4} the right-hand side is bounded by
\begin{align}
  \lesssim \, & K^{\epsilon^4p}\sum_\tau \left[C_\epsilon (KM)^{-\epsilon^2}R^\epsilon\|f_\tau\|_2\right]^p \notag \\
  \leq \,& K^{(\epsilon^4-\epsilon^2)p} \left[C_\epsilon M^{-\epsilon^2}R^\epsilon\|f\|_2\right]^p \notag
\end{align}
We choose $K=K(\epsilon)$ large so that $K^{(\epsilon^4-\epsilon^2)} \ll 1$, 
which yields by induction that the term involving $B_\epsilon$ plays 
an unimportant role. \\

For points $(x,t)$ not in $B_\epsilon$, we have the following decomposition into a transverse term and a bilinear tangent term.

\begin{lemma} \label{L:wall}
  For each point $(x,t)\in B_j\cap W$ satisfying $\max_{\tau} |e^{it\Delta}f_\tau(x)|\leq K^{-\epsilon^4}|e^{it\Delta}f(x)|$,
there exists a sub-collection $I$ of the collection of all possible $1/(KM)$
balls $\tau$, such that
  \begin{equation}
    |e^{it\Delta}f(x)| \lesssim |e^{it\Delta}f_{I,j,\rm trans}(x)| + K^{10}
{\rm Bil}(e^{it\Delta}f_{j, \rm tang}(x)),
  \end{equation}
  where
  \[
    f_{I,j,\rm trans}(x):=\sum_{\tau \in I} f_{\tau, j, \rm trans}(x),
  \]
  and the bilinear tangent term is given by
  \[
    {\rm Bil}(e^{it\Delta}f_{j,\rm tang}(x)):=\max_{\substack{\tau_1,\tau_2\\ \text{dist}(\tau_1,\tau_2)\geq 1/(KM)}}
|e^{it\Delta}f_{\tau_1,j,\rm tang}(x)|^{1/2}|e^{it\Delta}f_{\tau_2,j,\rm tang}(x)|^{1/2}.
  \]
\end{lemma}

\begin{proof}
Let $I$ be defined by $I:=\{\tau\,:\,|e^{it\Delta}f_{\tau,j,\rm tang}(x)|\leq K^{-10}|e^{it\Delta}f(x)| \}$.
  Then clearly
  \[
  I^c=\{\tau\,:\,|e^{it\Delta}f_{\tau,j,\rm tang}(x)|> K^{-10}|e^{it\Delta}f(x)| \}.
  \]
  If there exist $\tau_1, \tau_2 \in I^c$ with $\text{dist}(\tau_1,\tau_2)\geq 1/(KM)$, then $|e^{it\Delta}f(x)|
\lesssim K^{10} \text{Bil}(e^{it\Delta}f_{j,\rm tang}(x))$. Otherwise, the number of balls $\tau$ in $I^c$ is $O(1)$, and
  \[
    \sum_{\tau\in I^c}|e^{it\Delta}f_\tau (x)|\leq CK^{-\epsilon^4}|e^{it\Delta}f(x)|\leq \frac{1}{10}|e^{it\Delta}f(x)|.
  \]
Hence, by the fact that $f=\sum_\tau f_\tau$ and the definition of $I$,
  \begin{align}
    \frac{9}{10}|e^{it\Delta}f(x)|&\leq |\sum_{\tau\in I}e^{it\Delta}f_\tau(x)| \notag \\
                                  &\lesssim |e^{it\Delta}f_{I,j,\rm tang}(x)| + |e^{it\Delta}f_{I,j,\rm trans}(x)| \notag \\
                                  &\leq CK^{-8} |e^{it\Delta}f(x)| + |e^{it\Delta}f_{I,j,\rm trans}(x)|, \notag
  \end{align}
  which implies that $|e^{it\Delta}f(x)|\lesssim |e^{it\Delta}f_{I,j,\rm trans}(x)|$.
\end{proof}

By Lemma \ref{L:wall} we can now estimate the wall contribution in (\ref{cell-W}) by
\begin{align}
  & \big\|\hichi_W\, e^{it\Delta} f (x)\big\|^p_{L^p_x L^q_t (\bB^*_R)} \notag \\
\lesssim & \big\|\hichi_{B_\epsilon}e^{it\Delta} f (x)\big\|^p_{L^p_x L^q_t (\bB^*_R)} \label{Beps} \\
+ & \sum_{j}\big\|\max_{I}\hichi_{B_j\cap W}\,|e^{it\Delta} f_{I,j,\rm trans} (x)|\big\|^p_{L^p_x L^q_t (\bB^*_R)}\label{tran0}\\
  +&K^{10p} \sum_j\big\|\hichi_{B_j\cap W}{\rm Bil} (e^{it\Delta} f_{j,\rm tang} (x))\big\|^p_{L^p_x L^q_t (\bB^*_R)} \,.\label{Bil}
\end{align}
\\

As we explained above, the first term (\ref{Beps}) obeys an acceptable bound by induction on $M$.  We now estimate the linear transverse term (\ref{tran0}). The term (\ref{tran0}) is dominated by
\begin{equation}\label{maxI}
 \sum_{j}\sum_{I\subseteq \mathcal T}\big\|\hichi_{B_j\cap W}\,e^{it\Delta} f_{I,j,\rm trans} (x)\big\|^p_{L^p_x L^q_t (\bB^*_R)}\, ,
\end{equation}
where $\mathcal T$ is the collection of all possible $1/(KM)$-balls in $B(\xi_0, 1/M)$, and the sum is taken over all
subsets of $\mathcal T$. Since there are at most $2^{K^2}$ $I$'s,  we apply \eqref{max5} in Theorem \ref{thm-4} with radius $R^{1-\delta}$ to obtain
\begin{equation}
 (\ref{maxI}) \leq \sum_{j}  2^{K^2}  \left[C_\epsilon M^{-\epsilon^2} R^{(1-\delta)\epsilon}\|f_{j,\rm trans}\|_2\right]^p\,,
\end{equation}
which is bounded by, using Lemma \ref{L:trans},
\begin{equation}
 2^{K^2} R^{O(\epsilon^4)-\delta\epsilon p}  \left[C_\epsilon M^{-\epsilon^2} R^{\epsilon}\|f\|_2\right]^p\,.
\end{equation}
Since $\delta=\epsilon^2$, it is clear that
$2^{K^2} R^{O(\epsilon^4)-\delta\epsilon p} < 1/100$ and so the induction on the transverse term closes. \\

It remains to estimate the bilinear tangent term (\ref{Bil}).
We state the result on the bilinear maximal estimate in this section, and prove it in Section \ref{2maxsml}.

\begin{proposition}\label{Bil-prop}
For $p>3$, the following maximal estimate of the bilinear tangent term holds, uniformly in $M$:
 \begin{equation}\label{bi-est}
 \left( \int_{B(0,R)} \underset{t:(x,t)\in W\cap B_j}{\text{sup}}\big|{\rm Bil}(e^{it\Delta}f_{j,\rm tang}(x))\big|^p dx \right)^{1/p}
  \leq\, C_\epsilon R^{\epsilon/2} \|f\|_2\,.
\end{equation}
\end{proposition}

Given Proposition \ref{Bil-prop}, we estimate the bilinear tangent term (\ref{Bil}) as follows, for any $q>1/\epsilon^4$,
\begin{align}
  &\big\|\hichi_{B_j\cap W}{\rm Bil} (e^{it\Delta} f_{j,\rm tang} (x))\big\|^p_{L^p_x L^q_t (\bB^*_R)} \notag \\
  \leq & R^{p/q} \int_{B(0,R)} \underset{t:(x,t)\in W\cap B_j}{\text{sup}}\big|{\rm Bil}(e^{it\Delta}f_{j,\rm tang}(x))\big|^p dx \notag \\
  \leq & R^{O(\delta)+\epsilon p/2} \|f\|_2^p\,. \notag
\end{align}
Hence Theorem \ref{thm-4} follows from Proposition \ref{Bil-prop} and the inductions. \\

\section{Variations on the Strichartz inequality using decoupling} \label{locref}
\setcounter{equation}0

In this section we obtain both linear and bilinear local refinements of the Strichartz inequality,
via the Bourgain-Demeter $l^2$-decoupling theorem \cite{BD}.
In Section \ref{2maxsml} we will use the bilinear refinement to prove the bilinear maximal
estimate in Proposition \ref{Bil-prop}.\\

For the bilinear tangent term in Proposition \ref{Bil-prop}, all wave packets are tangent to a variety.  Suppose that $Z = Z(P)$ where $P$ is a product of non-singular polynomials.  For any tile $(\theta,\nu) \in\bT$,
we say that $T_{\theta,\nu}$ is \emph{$E R^{-1/2}$-tangent} to $Z$ if

$$T_{\theta,\nu}\subset N_{E R^{1/2}}Z \cap \bB^*_R, and$$

\begin{equation}
 \text{Angle}(G_0(\theta),T_z[Z(P)])\leq E R^{-1/2}
\end{equation}
for any non-singular point $z\in N_{2 E R^{1/2}} ( T_{\theta,\nu}) \cap 2\bB^*_R\cap Z$.

Let
\[
\bT_Z (E):=\{(\theta,\nu)\,|\,T_{\theta,\nu} \text{ is $E R^{-1/2}$-tangent to}\, Z\}\,,
\]
and  we say that $f$ is concentrated in wave packets from $\bT_Z(E)$ if
\[
 \sum_{(\theta,\nu)\notin \bT_Z(E)} \|f_{\theta,\nu}\|_2 \leq {\rm RapDec}(R)\|f\|_2.
\]

\noindent Since the radius of $T_{\theta, \nu}$ is $R^{1/2 + \delta}$, $R^\delta$ is the smallest interesting value of $E$.  

In this section, we establish the following local refinements of the Strichartz estimates.

\begin{theorem} \label{refstrich} Suppose that $f$ has Fourier support in $B^2(0,1)$, and is
concentrated in wave packets from $\bT_Z(E)$, where $Z=Z(P)$ and $P$ is a product of distinct non-singular polynomials.  Suppose that $Q_1, Q_2, ...$
are lattice  $R^{1/2}$-cubes in $B^3(R)$, so that
$$ \| \eit f \|_{L^6(Q_j)} \textrm{ is essentially constant in $j$}. $$
\noindent Suppose that these cubes are arranged in horizontal strips of the form $\ZR \times \ZR \times \{t_0, t_0 + R^{1/2} \}$, and that each strip contains $\sim \sigma$ cubes $Q_j$.  Let $Y$ denote $\bigcup_j Q_j$. Then
\begin{equation}
\label{linref}\| \eit f \|_{L^6(Y)} \lessapprox E^{O(1)} R^{-1/6} \sigma^{-1/3} \| f \|_{L^2}.
\end{equation}
\end{theorem}

To get some intuition, we consider a special case of Theorem \ref{refstrich}, in which the variety $Z$ is
naturally replaced by a 2-plane $V$, and $E \approx 1$.  In the planar case, all wave packets are contained in the
$\approx R^{1/2}$-neighborhood of $V$, and the absolute value $|\eit f(x)|$ is essentially constant along a
certain direction which is roughly normal to $V$. Note that $\eit f(x)|_V$ is a Schr\"odinger solution in
dimension 2. Denote $\eit f(x)|_V$ by $\eir h(y)$ for some function $h$ with Fourier support in $B^1(1)$, where $(y,r)$ are coordinates of $V$.
Hence the conclusion in Theorem \ref{refstrich} can be rephrased in terms of $h$. Indeed, observe that
 $$\| \eit f (x)\|^6_{L^6(Y)} \sim R^{1/2} \| \eir h(y)\|^6_{L^6(Y\cap V)},$$
  $$ \|f\|^2_2\sim R^{-1}\|\eit f\|_{L^2(B^3(R))}^2 \sim R^{-1}R^{1/2}\|\eir h\|^2_{L^2(B^3_R\cap V)}\sim R^{1/2}
\|h\|_2^2.$$
Therefore the estimate \eqref{linref} is equivalent to
\begin{equation}
\|\eir h\|_{L^6(Y\cap V)} \lessapprox \sigma^{-1/3}\|h\|_{L^2}.
\end{equation}
It follows from the Strichartz inequality that $\| \eir h \|_{L^6(Y\cap V)} \lesssim \| h \|_{L^2}$.  We get an
improvement when $\sigma$ is large.  The condition that $\sigma$ is large forces the solution $\eit f$ to be spread out in space, and we will exploit this spreading
out to get our improvement.

Moreover, Theorem \ref{refstrich} has the following bilinear refinement.

\begin{theorem} \label{bilrefstrich} For functions $f_1$ and $f_2$ with separated Fourier supports in $B^2(0,1)$,
separated by $\sim 1$, suppose that $f_1$ and $f_2$ are concentrated in wave packets from $\bT_Z(E)$, where $Z=Z(P)$ and $P$ is a product of distinct non-singular polynomials.
Suppose that $Q_1, Q_2, \cdots, Q_N$ are lattice  $R^{1/2}$-cubes in $B^3(R)$, so that for each $i$,
$$ \| \eit f_i \|_{L^6(Q_j)} \textrm{ is essentially constant in $j$}. $$
Let $Y$ denote $\bigcup_{j=1}^{N} Q_j$. Then
$$ \left\| |e^{i t \Delta} f_1 e^{i t \Delta} f_2 |^{1/2} \right\|_{L^6(Y)} \lessapprox E^{O(1)} R^{-1/6} N^{-1/6} \| f_1 \|_{L^2}^{1/2}\| f_2 \|_{L^2}^{1/2}. $$
\end{theorem}

\subsection{Proof of Theorem {\ref{refstrich}}}
 The proof uses the Bourgain-Demeter $l^2$-decoupling theorem, together with induction on the radius
and parabolic rescaling. First we recall the decoupling result of Bourgain and Demeter in \cite{BD}.
\begin{theorem}[Bourgain-Demeter]  \label{bourdem} Suppose that the $R^{-1}$-neighborhood of the unit parabola in
$\ZR^2$ is divided into $R^{1/2}$ disjoint rectangular boxes $\tau$, each with dimensions $R^{-1/2} \times R^{-1}$.
Suppose $\widehat F_\tau$ is supported in $\tau$ and $F = \sum_\tau F_\tau$.
Then
$$ \| F \|_{L^6(\ZR^2)} \lessapprox \left( \sum_\tau \| F_\tau \|_{L^6(\ZR^2)}^2 \right)^{1/2}. $$
\end{theorem}

If $E \ge R^{1/4}$ (or any fixed power of $R$), then the estimate (\ref{linref}) is trivial because of the factor $E^{O(1)}$.  So we assume that $E \le R^{1/4}$.  

To set up the argument, we decompose $f$ as follows.  We break the unit ball $B^2(1)$ in frequency space into
small balls $\tau$ of radius $R^{-1/4}$, and divide the physical space ball $B^2(R)$ into balls $B$ of
radius $R^{3/4}$.  For each pair $(\tau, B)$, we let $f_{\Box_{\tau,B}}$ be the function formed by cutting
 off $f$ on the ball $B$ (with a Schwartz tail) in physical space and the ball $\tau$ in Fourier space.
We note that $e^{it\Delta} f_{\Box_{\tau,B}}$, restricted to $B^3(R)$, is essentially supported on an
 $R^{3/4} \times R^{3/4} \times R$-box, which we denote by $\Box_{\tau,B}$ (compare the discussion in Section \ref{WPD}). The box $\Box_{\tau,B}$ is in the direction given by $(-2c(\tau),1)$ and intersects ${t=0}$ at a disk centered at $(c(B),0)$, where $c(\tau)$
and $c(B)$ are the centers of $\tau$ and $B$ respectively. For a fixed $\tau$, the different boxes $\Box_{\tau, B}$
tile $B^3(R)$.  In particular, for each $\tau$, a given cube $Q_j$ lies in exactly one box $\Box_{\tau, B}$.

Since $f$ is concentrated in wave packets from $\bT_Z(E)$, we only need to consider those $R^{1/2}$-cubes $Q_j$ that are contained in the $E R^{1/2}$-neighborhood of $Z$. For each such $R^{1/2}$-cube $Q_j$, we will see that the wave packets that pass through $Q_j$ are nearly coplanar.
Because of this, we will be able to apply the 2-dimensional decoupling theorem to study $e^{i t \Delta} f$ on $Q_j$:

\begin{lemma} \label{deconQj} Suppose that $f$ has Fourier support in $B^2(0,1)$ and is
concentrated in wave packets from $\bT_Z(E)$, where $E \le R^{1/4}$ and $Z=Z(P)$ is a finite union of non-singular varieties.  Suppose that an $R^{1/2}$-cube $Q$ is in $N_{E R^{1/2}}(Z)$.  Then we have the decoupling bound

\begin{equation}
\| e^{i t \Delta} f \|_{L^6(Q)} \lessapprox \left( \sum_\Box \| e^{ i t \Delta} f_\Box \|_{L^6(10Q)}^2 \right)^{1/2} + R^{-1000} \| f \|_{L^2}.
\end{equation}
\end{lemma}

Remark: The $R^{-1000} \| f \|_{L^2}$ is a negligibly small term which covers minor contributions coming from the tails of the Fourier transforms of smooth functions.  We will neglect this term in the sequel.

\begin{proof}
Observe that $Q \subset N_{E R^{1/2}}Z$ implies that there exists a non-singular point $z_0 \in Z \cap N_{E R^{1/2}}Q$. Thus for each wave packet $T_{\theta,\nu}$ that intersects $Q$, we have $z_0 \in Z \cap N_{2 E R^{1/2}}( T_{\theta,\nu})$.  By the definition of $\bT_{Z}(E)$ we get the angle bound
\begin{equation} \label{eq:Q}
  \text{Angle}(G_0(\theta),T_{z_0}[Z(P)])\leq E R^{-1/2}\,.
\end{equation}

We recall from Section \ref{WPD} that $G_0(\theta) = (- 2 c(\theta), 1)$.  Suppose that $T_{z_0} Z$ is the plane given by $a_1 x_1 + a_2 x_2 + b t = 0$, with $a_1^2 + a_2^2 + b^2 = 1$.  The angle condition above restricts the location of $\theta$ as follows:

\begin{equation} \label{locationtheta} | -2 a \cdot c(\theta) + b | \lesssim E R^{-1/2}. \end{equation}

\noindent We note that each tube $T_{\theta, \nu}$ makes an angle $\gtrsim 1$ with the plane $t=0$, because $\theta \subset B(0,1)$.  We can assume that there are some tubes $T_{\theta, \nu}$ tangent to $T_{z_0} Z$, and so $|a| \gtrsim 1$.  Therefore, (\ref{locationtheta}) confines $\theta$ to a strip of width $\sim E R^{-1/2}$ inside of $B(0,1)$.  We denote this strip by $S \subset B(0,1)$.

Let $\bT_{Z, Q}(E)$ be the set of $(\theta, \nu)$ in $\bT_Z(E)$ for which each $T_{\theta,\nu}$ intersects $Q$.  For each $(\theta, \nu)$ in $\bT_{Z, Q}(E)$, $\theta$ obeys (\ref{locationtheta}), and so $\theta \subset S$.  Let $\eta$ be a smooth bump function which approximates $\hichi_{Q}$.  We note that $\eta e^{i t \Delta} f$ is essentially equal to

$$ \sum_{(\theta, \nu) \in \bT_{Z, Q}(E) } \eta e^{i t \Delta} f_{\theta, \nu}. $$

\noindent Therefore, the Fourier transform of the localized solution $\eta e^{i t \Delta} f$ is essentially supported in

\begin{equation} \label{foursuppj} S^* := \{ (\xi_1, \xi_2, \xi_3): (\xi_1, \xi_2) \in S \textrm{ and } | \xi_3 - \xi_1^2 - \xi_2^2 | \lesssim R^{-1/2} \}. \end{equation}

(The contribution of the not essential parts is covered by the negligible term $R^{-1000} \| f \|_{L^2}$ in the statement of the Lemma.)

After a rotation in the $(x_1,x_2)$-plane we can suppose that the strip $S$ is defined by

$$ a_1 \le \xi_1 \le a_1 + E R^{-1/2}, $$

\noindent for some $a_1 \in [-1, 1]$.  We note that at each point $(\xi_1, \xi_2) \in S$,

\begin{equation} \label{par1par} \partial_1 \left( \xi_1^2 + \xi_2^2 \right) = 2 a_1 + O( E R^{-1/2}). \end{equation}

Let $v$ be the vector

$$ v = (1, 0, 2 a_1). $$

\noindent Let $\Pi$ be a 2-plane perpendicular to $v$.  Because $E \le R^{1/4}$, we claim that the projection of $S^*$ onto $\Pi$ lies in the $\sim R^{-1/2}$-neighborhood of a parabola.  
We can see this as follows.  Let

$$ S^*_{core} := \{ (\xi_1, \xi_2, \xi_3): \xi_1 = a_1, |\xi_2| \le 1, \xi_3 = \xi_1^2 + \xi_2^2 \}. $$

\noindent The set $S^*_{core}$ is a parabola, and its projection onto $\Pi$ is also a parabola.  We claim that the projection of $S^*$ to $\Pi$ lies in the $\sim R^{-1/2}$-neighborhood of this parabola.  If $(\xi_1, \xi_2, \xi_3) \in S^*$, then (\ref{par1par}) tells us that 

$$ (\xi_1^2 + \xi_2^2) = a_1^2 + \xi_2^2 + 2 a_1 (\xi_1 - a_1) + O( E R^{-1/2} \cdot | \xi_1 - a_1 | ). $$

\noindent Therefore,

$$ (\xi_1, \xi_2, \xi_3) = (a_1, \xi_2, a_1^2 + \xi_2^2) + (\xi_1 - a_1) v + O(E R^{-1/2} |\xi_1 - a_1| + R^{-1/2}). $$

\noindent The first term on the right-hand side lies is $S^*_{core}$.  Since $\Pi$ is perpendicular to $v$, the projection to $\Pi$ kills the second term on the right-hand side.  So the distance from the projection of $\xi$ to the projection of $S^*_{core}$ is at most

$$E R^{-1/2} |\xi_1 - a_1| + R^{-1/2} \lesssim E^2 R^{-1} + R^{-1/2} \sim R^{-1/2}. $$

Therefore, if we restrict $\eta e^{i t \Delta} f$ to $\Pi$, the resulting 2-dimensional function has Fourier support in the $\sim R^{-1/2}$-neighborhood of a parabola.

We consider the decomposition $f = \sum_{(\tau, B): \Box_{\tau, B} \cap Q \not= \emptyset} f_{\Box_{\tau, B}}$.  If $\eit f_{\Box_{\tau,B}}$ contributes to $\| \eit f \|_{L^6(Q)}$, there must be a wave
packet $T_{\theta,\nu}$ that intersects the $R^{1/2}$-cube $Q$ with $\theta \subset \tau$, and so $\tau \cap S$ must be non-empty.  Also, for a given $\tau$, there is only one $B$ so that $\Box_{\tau, B} \cap Q$ is non-empty.  Also, the Fourier support of $\eta e^{i t \Delta} f_{\Box_{\tau, B}}$ lies in $S^* \cap (\tau \times \ZR)$, by the same argument we used above for $\eta e^{i t \Delta} f$.  The projection onto $\Pi$ of $S^* \cap (\tau \times \ZR)$ is an $R^{-1/4} \times R^{-1/2}$ rectangular box.  The union of these boxes over all $\tau$ intersecting $S$ is the $R^{-1/2}$-neighborhood of a parabola.  Therefore, we have the hypotheses to apply the 2-dimensional decoupling theorem, Theorem \ref{bourdem}, which gives:

$$ \| \eta e^{i t \Delta} f \|_{L^6(\Pi)} \lessapprox \left( \sum_\Box \| \eta e^{ i t \Delta} f_\Box \|_{L^6(\Pi)}^2 \right)^{1/2}.  $$

Now we integrate in the direction perpendicular to $\Pi$ and apply Fubini and Minkowski to get

$$ \| \eta e^{i t \Delta} f \|_{L^6(\ZR^3)} \lessapprox \left( \sum_\Box \| \eta e^{ i t \Delta} f_\Box \|_{L^6(\ZR^3)}^2 \right)^{1/2}.  $$

This implies the desired conclusion.
\end{proof}

Next, by induction on the radius $R$, we will show that each function $f_\Box$ obeys a version of Theorem \ref{refstrich}.  Here is the statement.  Suppose that $S_1, S_2, ...$ are $R^{1/2} \times R^{1/2} \times R^{3/4}$-tubes in $\Box$ (running parallel to the long axis of $\Box$), and that
$$ \| e^{i t \Delta} f_\Box \|_{L^6(S_j)} \textrm{ is essentially constant in $j$}. $$
Suppose that these tubes are arranged into $R^{3/4}$-strips running parallel to the short axes of
$\Box$ and that each such strip contains $\sim \sigma_\Box$ tubes $S_j$.  Let $Y_\Box$ denote $\cup_j S_j$.  Then

\begin{equation} \label{indpar} \| e^{i t \Delta} f_\Box \|_{L^6(Y_\Box)} \lessapprox  E^{O(1)} R^{-1/12} R^{-1/12} \sigma_\Box^{-1/3} \| f_\Box \|_{L^2}. \end{equation}

This inequality follows by doing a parabolic rescaling and then using Theorem \ref{refstrich} at scale $R^{1/2}$, which we can assume holds by induction on $R$.  We write down the details of this parabolic rescaling, and in particular we will check that the tangent-to-variety condition is preserved under parabolic rescaling.
For each $R^{-1/4}$-ball $\tau$ in $B^2(1)$, we write $\xi=\xi_0+R^{-1/4}\zeta \in \tau$, then
$$|\eit f_\tau (x)| =R^{-1/4} |e^{i\tilde t \Delta} g (\tilde x)|$$
for some function $g$ with Fourier support in $B^2(1)$ and $\|g\|_2=\|f_\tau\|_2$, where the new coordinates $(\tilde x,\tilde t)$ are related to the old coordinates $(x,t)$ by
\begin{equation} \label{coord}
\begin{cases}
   \tilde x =R^{-1/4} x + 2 t R^{-1/4} \xi_0\,, \\
   \tilde t = R^{-1/2} t \,.
\end{cases}
\end{equation}
Therefore
$$\|\eit f_\Box (x)\|_{L^6(Y_\Box)} = R^{-1/12} \|e^{i\tilde t \Delta} g(\tilde x)\|_{L^6(\tilde Y)},$$
where $\tilde Y$ is the image of $Y_\Box$ under the new coordinates. Note that $\tilde Y$ is a union of $R^{1/4}$-cubes inside an $R^{1/2}$-cube. These $R^{1/4}$-cubes are arranged in $R^{1/4}$-horizontal strips, and each strip contains $\sim \sigma_\Box$ $R^{1/4}$-cubes. Moreover, by the relation \eqref{coord}, we see that each wave packet $T$, at scale $R$, of dimensions $R^{1/2+\delta}\times R^{1/2+\delta} \times R$ in the old coordinates is mapped to a corresponding wave packet $\tilde T$, at scale $R^{1/2}$, of dimensions $R^{1/4+\delta} \times R^{1/4+\delta} \times R^{1/2}$ in the new coordinates. The variety $Z(P)$ corresponds to a new variety $Z(Q)$,
given by the relation $Q(\tilde x,\tilde t) = Q(R^{-1/4} x + 2 t R^{-1/4} \xi_0, R^{-1/2} t) = P(x,t)$.
We claim that, under the above correspondence,
if the wave packet $T$ at scale $R$ is $E R^{-1/2}$-tangent to $Z(P)$, then the wave packet $\tilde T$ at scale
$R^{1/2}$ is $E R^{-1/4}$-tangent to $Z(Q)$ in the new coordinates.

By the relation \eqref{coord}, the distance condition $T \subset N_{E R^{1/2}}Z(P) $ implies that
$\tilde T \subset N_{E R^{1/4}}Z(Q)$.
Given the direction $(-2\xi,1)$ of $T$, the angle condition $$\text{Angle}((-2\xi,1),T_{z_0}[Z(P)])\leq E R^{-1/2}$$ is equivalent to
\begin{equation}\label{angle} \frac{|(-2\xi,1)\cdot (P_x (x_0,t_0), P_t (x_0,t_0))|}{|(P_x (x_0,t_0), P_t (x_0,t_0))|} \lesssim E R^{-1/2},\end{equation}
where $z_0=(x_0,t_0)$. Note that the direction of the corresponding wave packet $\tilde T$ is given by $(-2\zeta,1)$, where $\xi$ and $\zeta$ are related by $\xi=\xi_0+R^{-1/4}\zeta$. Let $\tilde z_0=(\tilde x_0,\tilde t_0)$ denote the point corresponding to $z_0$. Using the relations $$P_x = R^{-1/4}Q_{\tilde x} , \quad P_t = 2R^{-1/4} \xi_0 \cdot Q_{\tilde x} + R^{-1/2} Q_{\tilde t}\,, $$
after some computation, \eqref{angle} yields that
$$ \frac{|(-2\zeta,1)\cdot (Q_{\tilde x} (\tilde x_0,\tilde t_0), Q_{\tilde t} (\tilde x_0,\tilde t_0))|}{|(Q_{\tilde x} (\tilde x_0,\tilde t_0), Q_{\tilde t} (\tilde x_0,\tilde t_0))|} \lesssim E R^{-1/4},$$
which implies that $$\text{Angle}((-2\zeta,1),\tilde T_{\tilde z_0}[Z(Q)])\leq E R^{-1/4}.$$
Therefore the tangent-to-variety condition is preserved under parabolic rescaling and the induction on radius is justified.

We have now established inequality (\ref{indpar}).  To apply this inequality, we need to identify a good choice of $Y_\Box$.  We do this by some dyadic pigeonholing.  For each $\Box$, we apply the following algorithm to regroup tubes in $\Box$.

\begin{enumerate}

\item We sort those $R^{1/2} \times R^{1/2} \times R^{3/4}$-tubes $S$'s contained in the box $\Box$
according to the order of magnitude of $\| \eit f_\Box \|_{L^6(S)}$, which we denote $\lambda$.
For each dyadic number $\lambda$, we use $\mathbb{S}_\lambda$ to stand for the collection of tubes
$S \subset \Box$ with $\| \eit f_\Box \|_{L^6(S)} \sim \lambda$.

\item For each $\lambda$, we sort the tubes $S \in \mathbb{S}_\lambda$ by looking at the number of such
 tubes in an $R^{3/4}$-strip.  For any dyadic number $\eta$, we let $\mathbb{S}_{\lambda, \eta}$ be the set of tubes $S \in \mathbb{S}_\lambda$ so that the number of tubes of $\mathbb{S}_\lambda$ in the $R^{3/4}$-strip containing $S$ is $\sim \eta$.

\end{enumerate}

\begin{figure}  [ht]
\centering
\includegraphics[scale=.3]{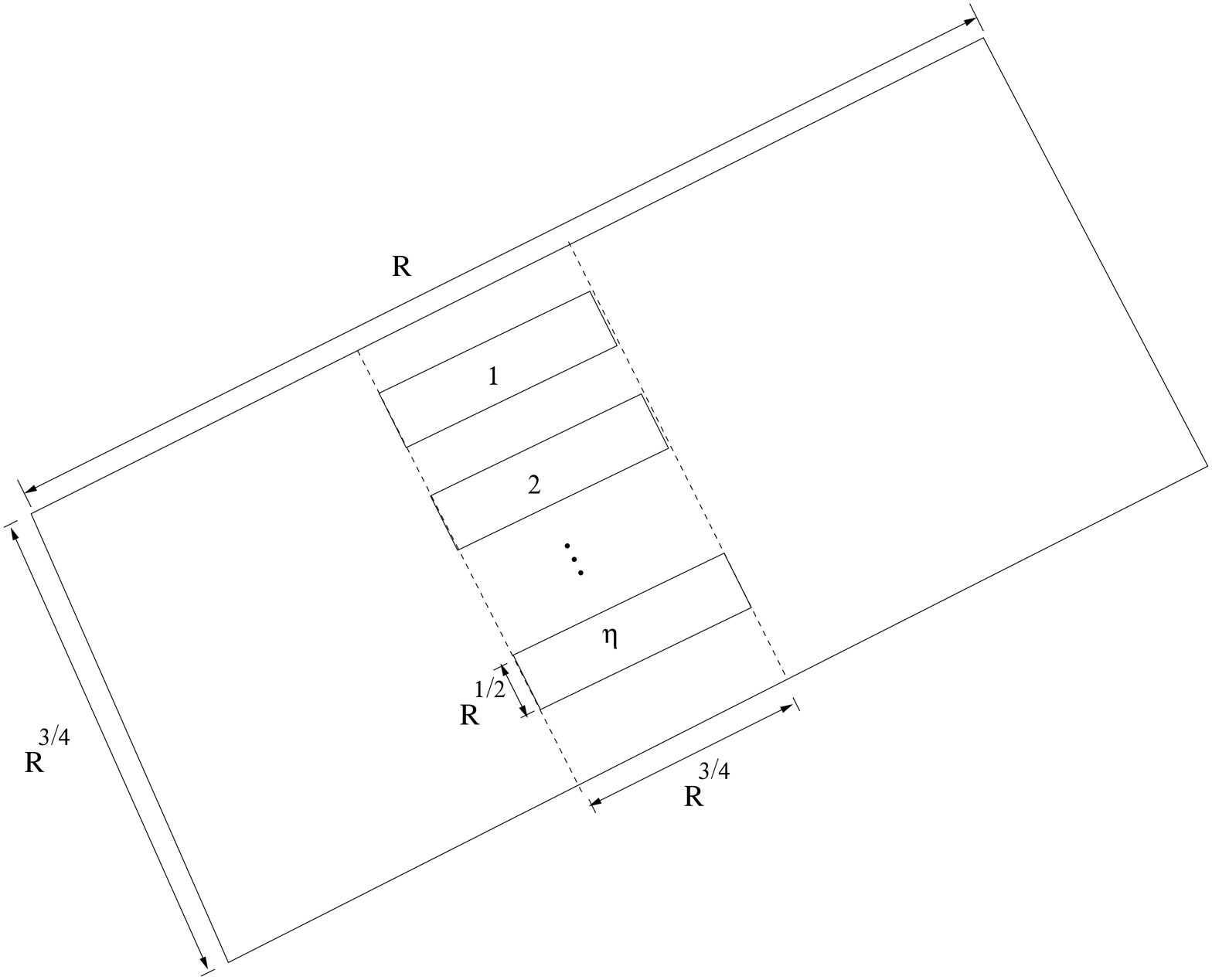}
\caption{\small{ Tubes in a given strip in the $\Box$  }}
\label{figure:Fig-box}
\end{figure}

Let $Y_{\Box, \lambda, \eta}$ be the union of the tubes in $\mathbb{S}_{\lambda, \eta}$. %and
%we use $\Id_{S}$ to denote the characteristic function of any measurable set $S$.
Then we represent
$$ \eit f = \sum_{\lambda, \eta} \left( \sum_\Box \eit f_\Box \cdot \hichi_{Y_{\Box, \lambda, \eta}} \right). $$

\noindent Note that $\| \eit f_\Box \|_{L^6(S)}\leq R^{O(1)}\|f\|_2$, for each tube $S$ as above
and the number of $\Box$'s does not exceed $R^{O(1)}$. 
We see that the contribution from those $\lambda$'s with $\lambda\leq R^{-C}\|f\|_2$ 
is at most $R^{-C/2}\|f\|_2$. Here the constant $C$ can be selected to be sufficiently large
so that $R^{-C/2}\|f\|_2$ is negligible. 
So without loss of generality, we can assume that the terms with small $\lambda$ contribute
insignificantly to $\| \eit f \|_{L^6(Q_j)}$ for every $Q_j$. 
Therefore there are only $O(\log R)$ significant choices for each of $\lambda, \eta$.  By pigeonholing,
we can choose $\lambda, \eta$ so that
\begin{equation} \label{dyadpig1}
\| \eit f \|_{L^6(Q_j)} \lesssim (\log R)^2 \big \| \sum_\Box \eit f_\Box \cdot\hichi_{Y_{\Box, \lambda, \eta}}\big \|_{L^6(Q_j)}
 \end{equation}
holds for a fraction $\approx 1$ of all cubes $Q_j$ in $Y$.
We need this uniform choice of $(\lambda, \eta)$, which is independent of $Q_j$,  because later we will sum over all $Q_j$ and arrive at $\|\eit f_\Box\|_{L^6(Y_{\Box,\lambda, \eta})}$. \\

We fix $\lambda$ and $\eta$ for the rest of the proof.
Let $Y_\Box$ stand for the abbreviation of $Y_{\Box, \lambda, \eta}$.
We note that $Y_\Box$ obeys the hypotheses for our inductive estimate (\ref{indpar}), with
 $\sigma_\Box$ being the value of $\eta$ that we have fixed.

The following geometric estimate will play a crucial role in our proof.  Each set $Y_\Box$ contains $\lesssim \sigma_\Box$ tubes in each strip parallel to the short axes of $\Box$.  Since the angle between the short axes of $\Box$ and the $x$-axes is bounded away from $\pi/2$, it follows that $Y_\Box$ contains $\lesssim \sigma_\Box$ cubes $Q_j$ in any $R^{1/2}$-horizontal row. Therefore,

\begin{equation} \label{geom1}  | Y_\Box \cap Y | \lesssim \frac{\sigma_\Box}{\sigma} |Y|. \end{equation}

Next we sort the the boxes $\Box$ according to the dyadic size of $\| f_\Box \|_{L^2}$.  We can restrict matters
to $\lesssim \log R$ choices of this dyadic size, and so we can choose a set of $\Box$'s, $\mathbb{B}$,
so that $\| f_\Box \|_{L^2}$ is essentially constant for $\Box \in \mathbb{B}$ and

\begin{equation} \label{dyadpig2}
\| \eit f \|_{L^6(Q_j)} \lessapprox  \| \sum_{\Box \in \mathbb{B}} \eit f_\Box\cdot \hichi_{Y_{\Box}} \|_{L^6(Q_j)}
\end{equation}
for a fraction $\approx 1$ of cubes $Q_j$ in $Y$.

Finally we sort the cubes $Q_j \subset Y$ according to the number of $Y_\Box$ that contain them.
 We let $Y' \subset Y$ be a set of cubes $Q_j$ which obey (\ref{dyadpig2}) and which each lies in $\sim \mu$ of
the sets $\{ Y_\Box \}_{\Box \in \mathbb{B}}$.  Because (\ref{dyadpig2}) holds for a large fraction of cubes,
and because there are only dyadically many choices of $\mu$, $|Y'| \approx |Y|$.  By the equation (\ref{geom1}), we
see that
$$  | Y_\Box \cap Y' | \le |Y_\Box \cap Y| \lessapprox  \frac{\sigma_\Box}{\sigma} |Y| \approx \frac{\sigma_\Box}{\sigma} |Y'|. $$
Therefore, the multiplicity $\mu$ is bounded by
\begin{equation}\label{geom2}
\mu \lessapprox \frac{\sigma_\Box}{\sigma} | \mathbb{B} |.
\end{equation}

We now are ready to combine all our ingredients and finish our proof.  For each $Q_j \subset Y'$, we have
$$ \| \eit f \|_{L^6(Q_j)} \lessapprox \left\| \sum_{\Box \in \mathbb{B}} \eit f_\Box \cdot\hichi_{Y_\Box} \right\|_{L^6(Q_j)}.  $$

Now we apply Lemma \ref{deconQj} to the function $\sum_{\Box \in \mathbb{B}, Q_j \in Y_\Box} f_\Box$ to bound the right hand side by

$$   \lessapprox \left(   \sum_{\Box \in \mathbb{B}, Q_j \subset Y_\Box} \left\| \eit f_\Box \right\|_{L^6(Q_j)}^2 \right)^{1/2}. $$

Since the number of $Y_\Box$ containing $Q_j$ is $\sim \mu$, we can apply H\"older to get
  $$  \left\| \sum_{\Box \in \mathbb{B}} \eit f_\Box \cdot \hichi_{Y_\Box} \right\|_{L^6(Q_j)} \lessapprox \mu^{1/3}
\left(   \sum_{\Box \in \mathbb{B}, Q_j \subset Y_\Box} \left\| \eit f_\Box \right\|_{L^6(Q_j)}^6 \right)^{1/6}. $$
Now we raise to the sixth power and sum over $Q_j \subset Y'$ to get
$$  \left\| \eit f \right\|_{L^6(Y')}^6 \lessapprox \mu^2 \sum_{\Box \in \mathbb{B}}  \left\| \eit f_\Box
\right\|_{L^6(Y_\Box)}^6. $$
Since $|Y'| \gtrapprox |Y|$, and since each cube $Q_j \subset Y$ makes an equal contribution to $\| \eit f \|_{L^6(Y)}$, we see that
$\| \eit f \|_{L^6(Y)} \approx \| \eit f \|_{L^6(Y')}$ and so
$$  \left\| \eit f \right\|_{L^6(Y)}^6 \lessapprox  \mu^2 \sum_{\Box \in \mathbb{B}}
\left\| \eit f_\Box \right\|_{L^6(Y_\Box)}^6. $$

By a parabolic rescaling, Figure \ref{figure:Fig-box} becomes Figure \ref{figure:Fig-cubes}.
\begin{figure}  [ht]
\centering
\includegraphics[scale=.4]{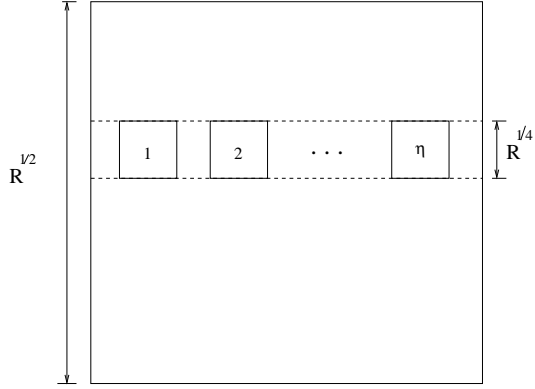}
\caption{\small{ Cubes in a given strip in an $R^{1/2}$-cube }}
\label{figure:Fig-cubes}
\end{figure}
Henceforth, applying our inductive hypothesis (\ref{indpar}) at scale $R^{1/2}$ to the right-hand side, we see that
\begin{equation} \label{forthm2}
 \left\| \eit f \right\|_{L^6(Y)}^6  \lessapprox E^{O(1)} R^{-1} \mu^2 \sigma_\Box^{-2} \sum_{\Box \in \mathbb{B}}
\| f_\Box \|_{L^2}^6.
\end{equation}
Plugging in our bound for $\mu$ in (\ref{geom2}), this is bounded by
$$  \lesssim E^{O(1)} R^{-1} \sigma^{-2} | \mathbb{B}|^2 \sum_{\Box \in \mathbb{B}} \| f_\Box \|_{L^2}^6. $$
Now since $\| f_\Box \|_{L^2}$ is essentially constant among all $\Box \in \mathbb{B}$, the last expression is
 $$ \sim E^{O(1)} R^{-1} \sigma^{-2} (\sum_{\Box \in \mathbb{B}} \| f_\Box \|_{L^2}^2)^3 \le E^{O(1)} R^{-1} \sigma^{-2}
 \| f \|_{L^2}^6. $$
 
\noindent Taking the sixth root, we obtain our desired bound: 
$$ \| \eit f \|_{L^6(Y)} \lessapprox E^{O(1)} R^{-1/6} \sigma^{-1/3} \| f \|_{L^2}. $$
This closes the induction on radius and completes the proof.

\subsection{Proof of Theorem {\ref{bilrefstrich}}}
It can be proved by the method used in the proof of Theorem \ref{refstrich}.
 By H\"older,
$$ \left\| \left|  \eit f_1 \eit f_2  \right|^{1/2}  \right\|_{L^6(Y)} \le \prod_{i=1}^2
\left\| \eit f_i  \right\|_{L^6(Y)}^{1/2}. $$
For each $i$, we process $\| \eit f_i \|_{L^6(Y)}$ following the proof of Theorem \ref{refstrich}.  We decompose
$f_i = \sum_\Box f_{i, \Box}$, and we follow the proof of Theorem \ref{refstrich}.   We define $Y_{i, \Box}$ by dyadic pigeonholing, so that $Y_{i, \Box}$ is arranged in several $R^{3/4}$-strips (running parallel to the short axes of $\Box$) with $\sim \sigma_{i, \Box}$ $R^{1/2} \times R^{1/2} \times R^{3/4}$-tubes in each strip.  When we use dyadic pigeonholing to pick a subset of cubes $Q_j \subset Y$, we pigeonhole for $f_1$ and $f_2$ simultaneously, and so we pick out a set of cubes that works well for both functions.
Following the argument up to Equation (\ref{dyadpig2}), we see that for a fraction $\approx 1$ of cubes $Q_j$,
\begin{equation} \label{dyadpig2'}
\| \eit f_i \|_{L^6(Q_j)} \lessapprox  \| \sum_{\Box \in \mathbb{B}_i} \eit f_{i, \Box} \cdot \hichi_{Y_{i, \Box}} \|_{L^6(Q_j)} \textrm{ for } i = 1, 2.
\end{equation}

Similarly, we sort the cubes $Q_j \subset Y$ according to the number of $Y_{i, \Box}$ that contain them.
 We let $Y' \subset Y$ be a set of cubes $Q_j$ which obey (\ref{dyadpig2'}) and which each lies in $\sim \mu_1$ of
the sets $\{ Y_{1, \Box} \}_{\Box \in \mathbb{B}_1}$ and $\sim \mu_2$ of the sets $\{ Y_{2, \Box} \}_{\Box \in \mathbb{B}_2}$.  Because (\ref{dyadpig2}) holds for a large fraction of cubes,
and because there are only dyadically many choices of $\mu_1, \mu_2$, $|Y'| \approx |Y|$.  Following the proof of Theorem \ref{refstrich} further, up to Equation \eqref{forthm2}, we see that for each $i$,

\begin{equation} \label{forthm2'}
 \left\| \eit f_i \right\|_{L^6(Y)}  \lessapprox E^{O(1)} R^{-1/6} \left[ \mu_i^2 \sigma^{-2}_{i, \Box}
\sum_{\Box \in \mathbb{B}_i} \| f_{i, \Box} \|_{L^2}^6 \right]^{1/6}.
\end{equation}

Finally, we give a geometric estimate for $\mu_1$ and $\mu_2$ that takes advantage of the bilinear structure.  If $\Box_1 \in \mathbb{B}_1$ and $\Box_2 \in \mathbb{B}_2$, then the
angle between their long axes is $\sim 1$.  Therefore, their intersection is contained in a ball of radius
$\sim R^{3/4}$, and so $Y_{\Box_1} \cap Y_{\Box_2}$ contains $\lesssim \sigma_{1, \Box} \sigma_{2, \Box}$
different $R^{1/2}$-balls (see Figure {\ref{figure:Fig-trans}}).
\begin{figure}  [ht]
\centering
\includegraphics[scale=.5]{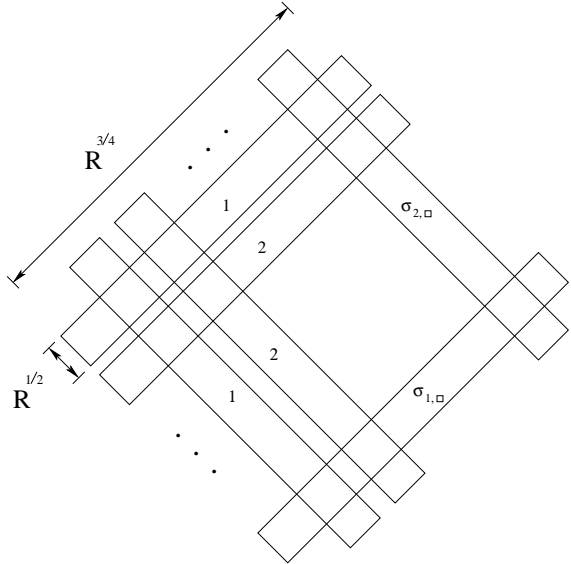}
\caption{\small{ at most $O(\sigma_{1, \Box}\sigma_{2, \Box})$ cubes created by two transversal families of
  rectangular boxes  }}
\label{figure:Fig-trans}
\end{figure}
For each of the $\approx N$ cubes $Q_j$ in $Y'$, for each $i$, the cube $Q_j$ lies
in $\sim \mu_i$ of the sets $\{ Y_{\Box_i} \}_{\Box_i \in \mathbb{B}_i}$.  Therefore,
\begin{equation} \label{geombil}
 N \prod_{i=1}^2 \mu_i \lessapprox \prod_{i=1}^2\sigma_{i, \Box}  | \mathbb{B}_i | .
\end{equation}
Starting with (\ref{forthm2'}) and inserting this estimate, we see that
$$\prod_{i=1}^2 \left\| \eit f_i  \right\|_{L^6(Y)}^{1/2} \lessapprox E^{O(1)} R^{-1/6}\prod_{i=1}^2 \left[ \mu_i^2 \sigma_{i, \Box}^{-2}
\sum_{\Box \in \mathbb{B}_i} \| f_{i, \Box} \|_{L^2}^6 \right]^{\frac{1}{6} \cdot \frac{1}{2} } $$
$$ \lessapprox E^{O(1)} R^{-1/6} \prod_{i=1}^2 \left[ N^{-1} | \mathbb{B}_i|^2  \sum_{\Box \in \mathbb{B}_i}
\| f_{i, \Box} \|_{L^2}^6 \right]^{\frac{1}{6} \cdot \frac{1}{2} } \lesssim E^{O(1)} R^{-1/6} N^{-1/6} \prod_{i=1}^2
\| f_i \|_{L^2}^{1/2}\,, $$
as desired.

\section{Bilinear maximal  estimate with small separation } \label{2maxsml}
\setcounter{equation}0

In this section, using Theorem \ref{bilrefstrich}  and parabolic rescaling, we prove the following proposition,
 which implies Proposition \ref{Bil-prop}.

 \begin{proposition} \label{2maxvar} 
 Suppose that $\xi_0 \in B^2(0,1)$ and that $f_i$ have Fourier supports in $B(\xi_0, 1/M)$ for some $ M \ge 1$.  Also suppose that the Fourier supports of $f_i$ are separated by at least $1/(KM)$, where $K = K(\eps)$ is a large constant.  Suppose that each $f_i$ is concentrated in wave packets from $\bT_Z(E)$, where $E \ge R^\delta$ and $Z=Z(P)$ and $P$ is a product of distinct non-singular polynomials. 
Then
\begin{equation} \label{2Lpmax}
\left\| |e^{i t \Delta} f_1|^{1/2}|e^{i t \Delta} f_2|^{1/2} \right\|_{L^3_x(B_R) L^\infty_t(0,R)} \lessapprox E^{O(1)} \| f_1 \|_{L^2}^{1/2}\| f_2 \|_{L^2}^{1/2}.
\end{equation}
\end{proposition}

\begin{proof} We can assume $M\ll R^{1/2}$, otherwise all wave packets were in the same direction and a direct computation would give us the desired result.

Since $f$ is concentrated in wave packets from $\bT_Z(E)$, we decompose $N_{E R^{1/2}}Z$ into balls $Q$ of
radius $R^{1/2}$.  Let $\eta$ be a smooth bump function approximating $\hichi_Q$.  As we saw in the proof of Lemma \ref{deconQj}, in Equation (\ref{foursuppj}), the Fourier support of each function $\eta e^{i t \Delta} f_i$ is essentially supported on
\begin{equation*} S^* := \{ (\xi_1, \xi_2, \xi_3): (\xi_1, \xi_2) \in S \textrm{ and } | \xi_3 - \xi_1^2 - \xi_2^2 | \lesssim R^{-1/2} \}, \end{equation*}

\noindent where $S \subset B(0,1)$ is a strip of width $E R^{-1/2}$.  Since the Fourier support of each $f_i$ is also contained in $B(\xi_0, 1/M)$, the Fourier support of $\eta e^{i t \Delta} f_i$ is also essentially contained in $B(\xi_0, \frac{2}{M}) \times \ZR$.  The intersection of $S^*$ with the cylinder $B(\xi_0, \frac{2}{M}) \times \ZR$ is contained in a rectangle of dimensions $\sim E R^{-1/2} \times 1/M \times 1/M$.  We denote this rectangle by $A^*(Q)$.  Since the Fourier support of each $\eta e^{i t \Delta} f_i$ is contained in $A^*(Q)$, 
$|\eta e^{i t \Delta} f_i|$ is morally constant on dual rectangles with dimensions $M \times M \times E^{-1} R^{1/2}$.  
We tile $Q$ with such dual rectangles, which we denote $A_k(Q)$.  The projection of each dual rectangle $A_k(Q)$ to the $x$-plane is an $M \times E^{-1} R^{1/2}$-rectangle.

Suppose that $\sup_{0 < t < R} | \eit f_1 \eit f_2|^{1/2} \sim H$ on a set $U \subset B(0,R)$.  It suffices for us to prove the bound
\begin{equation} \label{HU13}
 H |U|^{1/3} \lessapprox E^{O(1)} \| f_1 \|_{L^2}^{1/2} \| f_2 \|_{L^2}^{1/2}.
\end{equation}

We will bound $|U|$ using the rectangles $A_k(Q)$.  For the time being, let us suppose that $|\eta e^{i t \Delta} f_i|$ is roughly constant on each $A_k(Q)$.  This is not quite rigorous, but useful for intuition.  On the next page, we will come back to this point and give a rigorous argument.  

There must be a collection of dual rectangles $A_k(Q_j)$ whose projections cover $U$ and so that $| \eit f_1 \eit f_2|^{1/2} \sim H$ on each dual rectangle.  We let $X$ denote the union of these dual rectangles.  Each $M \times M \times E^{-1} R^{1/2}$ rectangle $A_k(Q_j) \subset X$ has a projection with area $M E^{-1} R^{1/2}$, and since these projections cover $U$, we have the bound

\begin{equation} \label{UvsX}
|U| \lesssim M^{-1} |X|. 
\end{equation}

We can also assume that no two rectangles $A_k(Q_j) \subset X$ have essentially the same projection.  This implies that $X$ contains $\lesssim E^{O(1)} R^{1/2} M^{-1}$ rectangles $A_k(Q)$ in each cube $Q$.  So for each cube $Q$, we get the bound

\begin{equation} \label{Xthin}
|X \cap Q| \lesssim E^{O(1)} M R.
\end{equation}

We consider the $R^{1/2}$-cubes $Q$ in $B^2(R) \times [0,R]$ that intersect $X$. We sort these $R^{1/2}$-cubes $Q$ according to the dyadic value of 
$\left\||\eit f_1|^{1/2}|\eit f_2|^{1/2}\right\|_{L^6(Q)}$. We can choose a set of of $R^{1/2}$-cubes $Q_j$, $j=1,2,\cdots, N$, so that
\begin{equation} \label{dya1}
\left\||\eit f_1|^{1/2}|\eit f_2|^{1/2}\right\|_{L^6(Q_j)} \text{is essentially constant in $j$},
\end{equation}
and $|X|\lessapprox |X\cap Y|$, where $Y:=\bigcup_{j=1}^{N} Q_j$.  Using the locally constant property that $|\eit f_1 \eit f_2|^{1/2} \sim H$ on each rectangle $A_k(Q_j) \subset X$, we see that 

\begin{equation} \label{locconst} H |X|^{1/6} \lessapprox E^{O(1)} \left\||\eit f_1|^{1/2}|\eit f_2|^{1/2}\right\|_{L^6(Y)}. \end{equation}

Since $|X \cap Q_j| \lesssim E^{O(1)} M R$ for each cube $Q_j$, $j=1, ... N$, we see that $|X| \lessapprox | X \cap Y| \lesssim E^{O(1)} M N R$.  Therefore,

\begin{equation}
H |X|^{1/3} \lessapprox E^{O(1)} M^{1/6} N^{1/6} R^{1/6}  \left\||\eit f_1|^{1/2}|\eit f_2|^{1/2}\right\|_{L^6(Y)}.
\end{equation}

Finally, since $|U| \lesssim M^{-1} |X|$, we have

\begin{equation}
H |U|^{1/3} \lessapprox E^{O(1)} M^{-1/6} N^{1/6} R^{1/6}  \left\||\eit f_1|^{1/2}|\eit f_2|^{1/2}\right\|_{L^6(Y)}.
\end{equation}

Therefore, our desired bound (\ref{HU13}) follows from a generalization of Theorem \ref{bilrefstrich}, which we now state.  

\begin{proposition} \label{bilrefstrichM}  Suppose that $f_1$ and $f_2$ are as in Proposition \ref{2maxvar}.  
Suppose that $Q_1, Q_2, \cdots, Q_N$ are lattice  $R^{1/2}$-cubes in $B^3(R)$ so that

\begin{equation} \label{dya2}
\left\||\eit f_1|^{1/2}|\eit f_2|^{1/2}\right\|_{L^6(Q_j)} \text{is essentially constant in $j$}.
\end{equation}

Let $Y$ denote $\bigcup_{j=1}^{N} Q_j$. Then
\begin{equation} \label{desiredbilinear}
 \left\||\eit f_1|^{1/2}|\eit f_2|^{1/2}\right\|_{L^6(Y)} \lessapprox E^{O(1)} M^{1/6} N^{-1/6} R^{-1/6} \| f_1 \|_{L^2}^{1/2} \| f_2 \|_{L^2}^{1/2}.  
\end{equation}
\end{proposition}

If $M=1$, then $f_1$ and $f_2$ have Fourier supports separated by $\sim 1$, and we can apply Theorem \ref{bilrefstrich}.  We first find $Y' \subset Y$ with $|Y'| \approx |Y|$ so that for each $i$, $\| \eit f_i \|_{L^6(Q_j)}$ is essentially constant among all $Q_j \subset Y'$.  Then we apply Theorem \ref{bilrefstrich} to $Y'$ to get (\ref{desiredbilinear}):

$$ \left\||\eit f_1|^{1/2}|\eit f_2|^{1/2}\right\|_{L^6(Y)} \approx $$

$$\approx \left\||\eit f_1|^{1/2}|\eit f_2|^{1/2}\right\|_{L^6(Y')} \lessapprox E^{O(1)} N^{-1/6} R^{-1/6} \| f_1 \|_{L^2}^{1/2} \| f_2 \|_{L^2}^{1/2}.  $$

For larger $M$, the Fourier supports of $f_1$ and $f_2$ are only separated by $\sim 1 / M$, and so we will need to apply parabolic rescaling before we can use Theorem \ref{bilrefstrich}.

Before we do this parabolic rescaling and prove Proposition \ref{bilrefstrichM}, let us return to the issue of $| e^{i t \Delta} f_i |$ being morally roughly constant on each rectangle $A_k(Q)$.  We used the locally constant property to justify (\ref{locconst}) above.  We can rigorously prove (\ref{locconst}) as follows.  We mentioned above that each function $\eta_Q e^{i t \Delta} f_i$ has Fourier transform essentially supported in a rectangle $A^*(Q)$ of dimensions $\sim E R^{-1/2} \times M^{-1} \times M^{-1}$.   So the Fourier transform of their product, $g := \eta_Q^2 e^{i t \Delta} f_1 e^{i t \Delta} f_2$, is essentially supported in a rectangle with the same orientation and roughly the same dimensions.  If $\hat \psi$ is designed to be identically 1 on this rectangle, then $g * \psi$ is essentially equal to $g$.  We can choose such a $\psi$ where $|\psi|$ is a rapidly-decaying approximation of $|A_k(Q_j)|^{-1} \hichi_{A_k(Q_j)}$.   Therefore, we see that 

\begin{equation} \label{supvsavg} \sup_{A_k(Q)} | \eit f_1 \eit f_2| \lesssim R^{O(\delta)} \frac{ \int_{R^\delta A_k(Q)}  | \eit f_1 \eit f_2| }{|A_k(Q_j)|} + R^{-1000} \| f_1 \|_{L^2} \| f_2 \|_{L^2}, \end{equation}

\noindent where the second term accounts for the tail of $\psi$.  Since $E \ge R^\delta$, we can assume that $R^\delta A_k(Q) \subset Q$.  

We let $X$ be a union of rectangles $A_k(Q_j)$ which each obeys 

$$H \lesssim \sup_{A_k(Q_j)} | \eit f_1 \eit f_2 |^{1/2}.$$

\noindent We can arrange that the projections of $10 A_k(Q_j)$ cover $U$ and also that any two rectangles $A_k(Q_j)$ in $X$ have essentially different projections.  Because of this covering, we still have $|U| \lesssim M^{-1} |X|$.  Now if $H \lesssim R^{-100} \| f_1 \|_{L^2}^{1/2} \| f_2 \|_{L^2}^{1/2}$, then (\ref{HU13}) follows trivially.  Therefore, (\ref{supvsavg}) tells us that for each $A_k(Q_j) \subset X$:

$$ \int_{R^\delta A_k(Q)}  | \eit f_1 \eit f_2| \gtrsim R^{-O(\delta)} |A_k(Q_j)| H^2. $$

We define $Y$ just as above, and this inequality lets us rigorously justify (\ref{locconst}):

$$ H |X|^{1/6} \approx H |X \cap Y|^{1/6} \lessapprox E^{O(1)} \left\||\eit f_1|^{1/2}|\eit f_2|^{1/2}\right\|_{L^6(Y)}.$$

It only remains to prove Proposition \ref{bilrefstrichM}.
 
\begin{proof}  For function $f$ with Fourier support in $B(\xi_0,1/M)$, by parabolic rescaling, we have \begin{equation} \label{parabresc}\|\eit f (x)\|_{L^p(B^3(R))} \sim M^{\frac 4 p -1}\|e^{ir\Delta}\tilde f (y)\|_{L^p(B_{R/M}\times I_{R/M^2})},\end{equation}
where $\tilde f$ has Fourier support in $B^2(0,1)$, $\|\tilde f\|_2=\|f\|_2$, the new coordinates $(y,r)$ and old coordinates $(x,t)$ are related by
$$\begin{cases} y=x/M +2t\xi_0/M, \\ r=t/M^2,\end{cases} $$
and $B_{R/M}\times I_{R/M^2}$ is a box of dimensions $\sim \frac R M \times \frac R M \times \frac{R}{M^2}$, which is the range for $(y,r)$ under the change of variables as above. By \eqref{parabresc}, we have
\begin{equation}\label{6resc}
  \left\||\eit f_1|^{1/2}|\eit f_2|^{1/2}\right\|_{L^6(Y)} \sim M^{-1/3} \left\||\eir \tilde f_1|^{1/2}|\eir \tilde f_2|^{1/2}\right\|_{L^6(\tilde Y)},
\end{equation}
where $\tilde f_1, \tilde f_2$ have $1/K$-separated Fourier supports in $B^2(0,1)$, and $\tilde Y$ is a union of $N$\, $\frac{\sqrt R}{M} \times \frac{\sqrt R}{M} \times \frac{\sqrt R}{M^2}$-boxes in $B_{R/M}\times I_{R/M^2}$, in correspondence to $Y$ under the change of variables as above.

To use Theorem \ref{bilrefstrich} to estimate $ \left\||\eir \tilde f_1|^{1/2}|\eir \tilde f_2|^{1/2}\right\|_{L^6(\tilde Y)}$, we decompose $B_{R/M}\times I_{R/M^2}$ as a union of $\frac{R}{M^2}$-balls $ Q_{k,R/M^2}$, and inside each $Q_{k,R/M^2}$ we consider the $\sqrt R/M$-cubes $Q^{(k)}$ that intersect $\tilde Y$. First, we sort the balls $Q_{k,R/M^2}$ according to the dyadic values $\|\eir \tilde f_i\|_{L^2(Q_{k,R/M^2})}$, $i=1,2$. Then inside each $Q_{k,R/M^2}$ we sort the cubes $Q^{(k)}$ according to the dyadic values $\|\eir \tilde f_i\|_{L^6(Q^{(k)})}$, $i=1,2$. We can choose balls $Q_{k,R/M^2}$, $k=1,2,\cdots, \bar W$, and inside each $Q_{k,R/M^2}$ we can choose a set of $\sqrt R/M$-cubes $Q^{(k)}_j$, $j=1,2,\cdots, N_k$, so that
\begin{equation}\label{dya3} \approx N\, \text{boxes in} \,\tilde Y\,\text{are contained in}\,\bigcup_{k=1}^{\bar W} \tilde Y_k, \end{equation} where $\tilde Y_k := \bigcup_{j=1}^{N_k}Q^{(k)}_j$, and the following conditions hold:
\begin{itemize}
\item (a). For each $i=1,2$, $\|\eir \tilde f_i\|_{L^2(Q_{k,R/M^2})}$ is essentially constant in $k=1,\cdots, \bar W$.

\item (b). For each $k=1,\cdots,\bar W$, for each $i=1,2$, $\|\eir \tilde f_i\|_{L^6(Q_j^{(k)})}$ is essentially constant in $j=1,\cdots, N_k$.

\item (c). $\left\||\eir \tilde f_1|^{1/2}|\eir \tilde f_2|^{1/2}\right\|_{L^6(\tilde Y_k)}$ is essentially constant in $k=1,\cdots, \bar W$.
\end{itemize}

Now by \eqref{dya2}, \eqref{dya3} and the condition (c) as above, for each $1\leq k \leq \bar W$ we have
$$\left\||\eir \tilde f_1|^{1/2}|\eir \tilde f_2|^{1/2}\right\|_{L^6(\tilde Y)} \lessapprox \bar W^{\frac 1 6} \left\||\eir \tilde f_1|^{1/2}|\eir \tilde f_2|^{1/2}\right\|_{L^6(\tilde Y_k)}.$$
Since tangent-to-variety condition is preserved under parabolic rescaling, we can apply Theorem \ref{bilrefstrich} to bound $ \left\||\eir \tilde f_1|^{1/2}|\eir \tilde f_2|^{1/2}\right\|_{L^6(\tilde Y_k)}$ by
$$  \lessapprox E^{O(1)} \left(\frac{R}{M^2}\right)^{-1/6}N_k^{-1/6} \left(\frac{R}{M^2}\right)^{-1/2} \prod_{i=1}^2 \left\|\eir \tilde f_i\right\|^{1/2}_{L^2(Q_{k,R/M^2})}.$$
By the condition (a) as above and parabolic rescaling \eqref{parabresc}, we have
$$\prod_{i=1}^2 \left\|\eir \tilde f_i\right\|^{1/2}_{L^2(Q_{k,R/M^2})} \lesssim \bar W^{-1/2} \prod_{i=1}^2 \|\eir \tilde f_i\|^{1/2}_{L^2(B_{R/M}\times I_{R/M^2})}  $$
$$\sim \bar W^{-1/2} M^{-1} \prod_{i=1}^2 \|\eit f_i\|^{1/2}_{L^2(B^3(R))} \lesssim \bar W^{-1/2} M^{-1} R^{1/2} \prod_{i=1}^2 \|f_i\|_2^{1/2}.$$
Combining \eqref{6resc} and the above estimates for $\left\||\eir \tilde f_1|^{1/2}|\eir \tilde f_2|^{1/2}\right\|_{L^6(\tilde Y)}$, we get
$$\left\||\eit f_1|^{1/2}|\eit f_2|^{1/2}\right\|_{L^6(Y)} \lessapprox E^{O(1)} \bar W^{-1/3}N_k^{-1/6}R^{-1/6}\prod_{i=1}^{2}\|f_i\|_2^{1/2}.$$
The above estimate holds for $\bar W$ indexes $k$'s. For each $k$, there are $N_k$\, $\frac{\sqrt R}{M}$-cubes in $\tilde Y_k$, each $\frac{\sqrt R}{M}$-cube contains
 at most $M$\, $\frac{\sqrt R}{M} \times \frac{\sqrt R}{M} \times \frac{\sqrt R}{M^2}$-boxes in $\tilde Y$, and there are $\approx N$\,
$\frac{\sqrt R}{M} \times \frac{\sqrt R}{M} \times \frac{\sqrt R}{M^2}$-boxes in $\tilde Y$ that are contained in $\bigcup_{k=1}^{\bar W} \tilde Y_k$. By pigeonholing there is an index $k$ satisfying
$$N \lessapprox N_k\bar W M.$$ Therefore
\begin{equation} \label{2L6'}
\left\||\eit f_1|^{1/2}|\eit f_2|^{1/2}\right\|_{L^6(Y)} \lessapprox E^{O(1)} \bar W^{-1/6} N^{-1/6} M^{1/6} R^{-1/6} \prod_{i=1}^2\|f_i\|_2^{1/2}\,.
\end{equation}
Since $\bar W \geq 1$, this completes the proof of Proposition \ref{bilrefstrichM}.
\end{proof}

This finishes the proof of Proposition \ref{2maxvar}. \end{proof} 

Finally, to prove Proposition \ref{Bil-prop}, we apply Proposition \ref{2maxvar} to $f_{j, tang}$ on each ball $B_j$.  We expand $f_{j, tang}$ into wave packets at the scale $\rho = R^{1 - \delta}$ on the ball $B_j$.  Because of the definition of $f_{j, tang}$, each wave packet will lie in the $R^{1/2 + \delta}$-neighborhood of $Z$ and the angles between the wave packets and the tangent space of $Z$ will be bounded by $R^{-1/2 + 2 \delta}$.  For a detailed description of the wave packet decomposition of $f_{j, tang}$ on a smaller ball, see Section 7 of \cite{lG16}.  We define $E$ so that $\rho^{1/2} E = R^{1/2 + \delta}$.  Since $\rho = R^{1 - \delta}$, we get $E = R^{(3/2) \delta}$, and so $E \rho^{-1/2} = R^{-1/2 + 2 \delta}$.  Each new wave packet lies in the $E \rho^{1/2}$-neighborhood of $Z$, and the angles between the wave packets and the tangent space of $Z$ are bounded by $E \rho^{-1/2}$.  Therefore, the new wave packets are concentrated in $\bT_Z(E)$.  Now since $E^{O(1)} = R^{O(\delta)}$, the bound from Proposition \ref{2maxvar} implies Proposition \ref{Bil-prop}.

\end{document}